\documentclass[10pt,reqno]{amsart}
\RequirePackage[OT1]{fontenc}
\RequirePackage{amsthm,amsmath}
\RequirePackage[numbers]{natbib}
\RequirePackage[colorlinks,linkcolor=blue,citecolor=blue,urlcolor=blue]{hyperref}
\usepackage{amssymb}
\usepackage{anysize}
\usepackage{hyperref}
\usepackage{extarrows}
\usepackage{multirow}
\usepackage{natbib}
\usepackage{stmaryrd}
\usepackage{color}
\usepackage{amsmath}
\usepackage{enumitem}
\usepackage{mathtools} 
\usepackage{dsfont} 
\usepackage{comment}
\usepackage{soul}

\numberwithin{equation}{section}
\theoremstyle{plain} 
\newtheorem{theorem}{Theorem}[section]
\newtheorem{lemma}[theorem]{Lemma}

\newtheorem{proposition}[theorem]{Proposition}

\theoremstyle{definition}
\newtheorem{definition}[theorem]{Definition}
\newtheorem{assumption}[theorem]{Assumption}

\theoremstyle{remark}

\renewcommand{\Im}{\mathrm{Im}\,}

\newcommand{\E}{{\mathbb E }}

\newcommand{\R}{{\mathbb R }}
\newcommand{\N}{{\mathbb N}}

\renewcommand{\P}{{\mathbb P}}
\newcommand{\C}{{\mathbb C}}

\newcommand{\ii}{\mathrm{i}}
\newcommand{\ee}{\mathrm{e}}
\newcommand{\deq}{\mathrel{\mathop:}=}
\newcommand{\dd}{\mathrm{d}}
\newcommand{\ie}{\emph{i.e., }}
\newcommand{\eg}{\emph{e.g., }}
\newcommand{\cf}{\emph{c.f., }}

\newcommand{\wt}{\widetilde}

\newcommand{\bs}{\boldsymbol}

\def\Tr{\mathrm{Tr}}
\def\i{\text{i}}

\def\Dim{\Delta {\mathrm{Im}}\,}

\def\nn{\mathfrak{n}}

\def\one{\mathds{1}}

\def\<{\langle}
\def\>{\rangle}

\def\X{\mathcal{X}}

\renewcommand{\mathbf}[1]{\bs{#1}}
 
\marginsize{25mm}{25mm}{25mm}{26mm}

\allowdisplaybreaks

\begin{document}

 \begin{minipage}{0.85\textwidth}
 	\vspace{2.5cm}
 \end{minipage}
 \begin{center}
 	\large\bf Quantitative Tracy--Widom laws for the largest eigenvalue\\ of generalized Wigner matrices
 	
 \end{center}

 \renewcommand{\thefootnote}{\fnsymbol{footnote}}	
 \vspace{0.5cm}
 
 \begin{center}
 	\begin{minipage}{1.4\textwidth}

 		\begin{minipage}{0.33\textwidth}
 			\begin{center}
 				Kevin Schnelli\footnotemark[1]\\
 				\footnotesize 
 				{KTH Royal Institute of Technology}\\
 				{Stockholm, Sweden}\\
 				{\it schnelli@kth.se}
 			\end{center}
 		\end{minipage}
 		\begin{minipage}{0.33\textwidth}
 			\begin{center}
 				Yuanyuan Xu\footnotemark[2]\\
 				\footnotesize 
 				{Institute of Science and Technology Austria}\\
 				{Klosterneuburg, Austria}\\
 				{\it yuanyuan.xu@ist.ac.at}
 			\end{center}
 		\end{minipage}
 	\end{minipage}
 \end{center}
 
 \bigskip

 \footnotetext[1]{Supported by the Swedish Research Council (VR-2017-05195, VR-2021-04703), and the Knut and Alice Wallenberg Foundation.}
 \footnotetext[2]{Supported by the Swedish Research Council Grant VR-2017-05195, and the ERC Advanced Grant ``RMTBeyond" No.~101020331}

 \renewcommand{\thefootnote}{\fnsymbol{footnote}}	
 
 \vspace{1cm}
 
 \begin{center}
 	\begin{minipage}{0.83\textwidth}\footnotesize{
 			{\bf Abstract.}}
 		We show that the fluctuations of the largest eigenvalue of any generalized Wigner matrix~$H$ converge to the Tracy--Widom laws at a rate nearly $O(N^{-1/3})$, as the matrix dimension~$N$ tends to infinity. We allow the variances of the entries of~$H$ to have distinct values but of comparable sizes such that $\sum_{i} \E|h_{ij}|^2=1$. Our result improves the previous rate $O(N^{-2/9})$ by Bourgade~\cite{Bourgade extreme} and the proof relies on the first long-time Green function comparison theorem near the edges without the second moment matching restriction.

 	\end{minipage}
 \end{center}

 \vspace{5mm}
 
 {\small
 	\footnotesize{\noindent\textit{Date}: August 3, 2022}
 	\footnotesize{\noindent\textit{MSC classes}: 15B52, 60B20}
 }
 
 \vspace{2mm}

 \thispagestyle{headings}

\section{Introduction}

In this paper we study quantitative statements of the edge universality for generalized Wigner matrices. The edge universality for self-adjoint random matrix models states that the appropriately centered largest eigenvalue $\lambda_N$ fluctuates on scale $N^{-2/3}$, with $N$ the matrix size, and the distributional convergence
\begin{align}\label{basic convergence}
	\lim_{N\rightarrow \infty} \P\big(N^{2/3}(\lambda_N-E_+)\le r\big)=\mathrm{TW}_{\beta}(r)\,,\qquad r\in\R\,,
\end{align}
holds with  $E_+$ being the centering. The universal limiting laws $\mathrm{TW}_\beta$ were identified by Tracy and Widom~\cite{TW1,TW2} who proved the above convergence with $E_+=2$ for the invariant Gaussian ensembles, GUE and GOE, with $\beta=2$ for the complex Hermitian respectively $\beta=1$ for the real symmetric symmetry class. Using the integrability structure of the Gaussian ensembles the convergence in~\eqref{basic convergence} was quantified explicitly by Johnstone and Ma in~\cite{gaussian_speed} as a Berry-Esseen type theorem:  For any fixed $r_0 \in \R$, there exists a constant $C=C(r_0)$ such that 
\begin{equation}\label{gaussian}
	\sup_{r > r_0}\Big|\P^{\mathrm{G\beta E}}\Big( N^{2/3} (\lambda_N-E_+ )<r \Big)- \mathrm{TW}_{\beta}(r) \Big| \leq C N^{-2/3}\,,
\end{equation}
for sufficiently large $N$, where $E_+=2$ for the GUE, respectively $E_+=\sqrt{4-\frac{2}{N}}$ for the GOE. Related asymptotic expansions of edge correlation kernels were obtained in \cite{choup, Forrester,peter2} and strong convergence results of edge kernel for $\beta$-ensembles with general potentials were derived in~\cite{Deift,Deift2}.

The edge universality, \ie the extension of the convergence results in~\eqref{basic convergence} to non-invariant matrix ensembles, was established in many works over the last two decades, \eg for Wigner matrices~\cite{rigidity,LY, PeSo2, So1, TV2} with deformations and generalizations~\cite{XXX,BEY edge universality,KY_deform, deformed}, for adjacency matrices of random graphs~\cite{huang_dregular, sparse0,He,HLY,sparse_huang,sparse_lee, sparse}, for band random matrices \cite{band_sodin}, and for sample covariance matrices~\cite{bao, Ding+Yang, sample_kevin, peche, PY1,wang_ke}. Yet results for effective convergence rates other than the invariant ensembles are still scarce. The first quantitative estimate for generalized Wigner matrices was obtained by Bourgade in~\cite{Bourgade extreme} who established a convergence rate of order almost $O(N^{-2/9})$. In our recent work~\cite{Schnelli+Xu} we obtained the convergence rate almost $O(N^{-1/3})$ for Wigner matrices. The purpose of the present work is to establish the same convergence rate for generalized Wigner matrices. Generalized Wigner matrices, introduced in~\cite{EYY1}, have independent centered entries, up to the symmetry constraints, whose distributions may be distinct with an inhomogeneous variance profile. More precisely, denoting the entries by $h_{ij}$ one has $\E h_{ij}=0$ and $\E |h_{ij}|^2=S_{ij}$ with $S=(S_{ij})$ a doubly stochastic matrix. In case $S_{ij}=1/N$ one recovers the usual definition of Wigner matrices. We will assume that the entries of $S$ are uniformly bounded from below implying a spectral gap; see~\eqref{S_spectrum} later. A theoretical motivation for studying this type of models is to extend the universality to models beyond mean-field systems. On the methodological side, we aim to lift the second moment matching restriction for comparisons in previous works, \eg~\cite{sparse0,rigidity,LY}, to prove edge universality. In applications generalized Wigner matrices arise as centered adjacency matrices of balanced stochastic block models~\cite{bickel,Ji_Oon,Lei}.

We derive our main convergence results using a {\it Green function comparison} strategy tracing back to Erd\H{o}s, Yau and Yin~\cite{EYY1}, see also the related four moment theorem of Tao and Vu~\cite{TV1}, for the bulk universality. There a four moment matching condition is required and this restriction can be removed using a sophisticated dynamical approach relying on the local relaxation of Dyson's Brownian motion~(DBM)~\cite{EYS} and related subsequent works in~\cite{BEYY, book, fixed_DBM, LandonYau}. The Green function comparison approach to the edge universality only requires second moment matching \cite{rigidity}, while earlier works on the edge universality use moment methods~\cite{PeSo2,SinaiSo,So1} or third moment matching~\cite{TV2}. Edge universality for generalized Wigner matrices was first proved in~\cite{BEY edge universality} by combining local relaxation results of the DBM at the edges (see also~\cite{AH,Bourgade extreme,LandonYau_edge}) and a short-time Green function comparison to remove the small Gaussian convolution. Combining his quantitative local relaxation estimates for the DBM with a Green function comparison for short times, Bourgade~\cite{Bourgade extreme} obtained the convergence rate $O(N^{-2/9})$ to the Tracy--Widom laws for generalized Wigner matrices.

In this paper, we use a long-time continuous Green function flow~\cite{deformed,sparse} in combination with cumulant expansions for our Green function comparison on a much finer spectral parameter scale (slightly above~$N^{-1}$) with an improved error estimate $O(N^{-1/3})$. The novelty of our method is to establish the first Green function comparison theorem near the edges without the second moment matching restriction, which is new even for Gaussian matrices. This comparison requires precise estimates on the contributions from inhomogeneous variances of the matrix entries to the interpolating Green function flow that turn out to be considerably harder than those from the third and fourth order moments considered in \cite{rigidity,Schnelli+Xu}. To overcome this difficulty, we introduce a new expansion mechanism for products of the Green function entries and the entries of the variance profile matrix $S$. Performing the cumulant expansions iteratively in combination with the spectral gap of $S$ in (\ref{S_spectrum}) below and applying a Gr\"onwall argument, we extend the explicitly-computable estimates for the Gaussian invariant ensembles \cite{gaussian_speed} to arbitrary generalized Wigner matrices.

{\it Organization of the paper:} In Section~\ref{sec:main_result}, we state the quantitative Tracy--Widom laws for generalized Wigner matrices which follow from our main technical result, the Green function comparison in Theorem~\ref{green_comparison}. In Section~\ref{sec:toy}, we consider a special case of Theorem~\ref{green_comparison} with $F=\mathrm{id}$, \ie Proposition~\ref{GCT_mn}. First, in Subsection~\ref{sec:strategy} we sketch the proof strategy for Proposition~\ref{GCT_mn}, and the formal proof is subsequently presented in Subsection~\ref{sec:proof_prop}. The proof contains two main ingredients, Proposition~\ref{GCT_mn_1} and Proposition~\ref{GCT_mn_2}, which are proved in Section~\ref{sec:step1} and Section~\ref{sec:step2} respectively. Finally, in Section~\ref{sec:proof} we give the proof of Theorem~\ref{green_comparison} for a general function $F$ by extending the proof scheme of Proposition~\ref{GCT_mn}.

{\it Acknowledgment:} We thank L\'aszl\'o Erd\H{o}s and Rong Ma for helpful discussions.\\

{\it Notation:} We will use the following definition on high-probability estimates from~\cite{Erdos+Knowles+Yau}. 
\begin{definition}\label{definition of stochastic domination}
	Let $\mathcal{X}\equiv \mathcal{X}^{(N)}$ and $\mathcal{Y}\equiv \mathcal{Y}^{(N)}$ be two sequences of nonnegative random variables. We say~$\mathcal{Y}$ stochastically dominates~$\mathcal{X}$ if, for all (small) $\tau>0$ and (large)~$\Gamma>0$,
	\begin{align}\label{prec}
		\P\big(\mathcal{X}^{(N)}>N^{\tau} \mathcal{Y}^{(N)}\big)\le N^{-\Gamma},
	\end{align}
	for sufficiently large $N\ge N_0(\tau,\Gamma)$, and we write $\mathcal{X} \prec \mathcal{Y}$ or $\mathcal{X}=O_\prec(\mathcal{Y})$.
\end{definition}
We often use the notation $\prec$ also for deterministic quantities, then~\eqref{prec} holds with probability one. Properties of stochastic domination can be found in the following lemma.
\begin{lemma}[Proposition 6.5 in \cite{book}]\label{dominant}
	\begin{enumerate}
		\item $X \prec Y$ and $Y \prec Z$ imply $X \prec Z$;
		\item If $X_1 \prec Y_1$ and $X_2 \prec Y_2$, then $X_1+X_2 \prec Y_1+Y_2$ and $X_1X_2 \prec Y_1Y_2;$
		\item If $X \prec Y$, $\E Y \geq N^{-c_1}$ and $|X| \leq N^{c_2}$ almost surely with some fixed exponents $c_1$, $c_2>0$, then we have $\E X \prec \E Y$.
	\end{enumerate}
\end{lemma}

For any matrix $A \in \C^{N \times N}$, the matrix norm induced by the Euclidean vector norm is denoted by $\|A\|:=\sigma_{\max}(A)$, where $\sigma_{\max}(A)$ denotes the largest singular value of $A$. We denote the sup norm of the matrix by $\|A\|_{\max}:=\max_{i,j}|A_{ij}|$. 

Throughout the paper, we use~$c$ and~$C$ to denote strictly positive constants that are independent of~$N$. Their values may change from line to line. For $X,Y \in \R$, we write $X \sim Y$ if there exist constants $c, C>0$ such that $c |Y| \leq |X| \leq C |Y|$ for large $N$. Finally, we denote $\C^+\deq\{z\in\C\,:\,\Im z>0\}$ and $\R^+\deq\{x\in\R\,:\,x \geq 0\}$.

\section{Main result}\label{sec:main_result}

Let $H=(h_{ij})_{1\leq i,j\leq N}$ be an $N \times N$ real symmetric ($\beta=1$) or complex Hermitian ($\beta=2$) generalized Wigner matrix satisfying the following assumptions.
\begin{assumption}\label{assumption_1}
	\begin{enumerate}
		\item $\{ h_{ij}|  i \leq j \} $ are independent random variables with $\E[h_{ij}] = 0$.
		\item Denoting the {\it variance profile matrix} by $S =(S_{ij})_{1\leq i,j\leq N}$ where $S_{ij}=\E |h_{ij}|^2$, then $S$ satisfies the  summation condition
				\begin{equation}\label{sum_S}
					\sum_{i=1}^N S_{ij}=1, \quad \qquad 1 \leq j \leq N.
				\end{equation}	
		Moreover, there exist two strictly positive constants $C_{\inf},C_{\sup}$ independent of $N$ such that
		\begin{equation}\label{flat}
			C_{\inf} \leq \inf_{i,j} \{N S_{ij} \}\leq \sup_{i,j}\{N S_{ij}\}  \leq C_{\sup}.
		\end{equation}
	For the complex case, we additionally assume that 
	\begin{equation}\label{condition_complex}
		\E [(h_{ij})^2]=0, \qquad  \forall~i \neq j.
	\end{equation}
		\item  All moments of the entries of $\sqrt{N}H$ are uniformly bounded, i.e., for any $k \geq 3$, there exists $C_k$ independent of $N$ such that for all $1 \leq i,j \leq N$,
		\begin{equation}\label{moment_condition}
			\E |\sqrt{N} h_{ij}|^k \leq C_k.
		\end{equation}	
	\end{enumerate}

\end{assumption}

Note that the variance profile matrix $S$ is a symmetric and doubly stochastic matrix whose spectrum lies in $[-1,1]$. More importantly, from the lower bound in~(\ref{flat}), there is a spectral gap in the spectrum of $S$, \ie there exist constants $c_\pm \geq C_{\inf}>0$ such that 
\begin{equation}\label{S_spectrum}
	\mathrm{Spec}(S) \subset [-1+c_-,1- c_+] \cup \{1\}\,,
\end{equation}
where $1$ is a simple eigenvalue. For a reference see Chapter 6.5 in \cite{book}.

 	In the homogeneous case $S_{ij} = 1/N$, we recover the original definition of Wigner matrices with the spectral gap $c_\pm=1$. 	The prominent Gaussian invariant ensembles are special Wigner matrices with independent Gaussian entries, which we denote by G$\beta$E for short. More precisely, for the Gaussian unitary ensemble (GUE, $\beta=2$), one requires $\sqrt{N}h_{ij} \stackrel{{\rm d}}{=} \mathbf{N}(0,1/2)+\ii \mathbf{N}(0,1/2)$ and $\sqrt{N}h_{ii} \stackrel{{\rm d}}{=} \mathbf{N}(0,1)$. For the Gaussian orthogonal ensemble (GOE, $\beta=1$), we assume $\sqrt{N}h_{ij} \stackrel{{\rm d}}{=} \mathbf{N}(0,1)$ ($i\not=j$) and $\sqrt{N}h_{ii} \stackrel{{\rm d}}{=} \mathbf{N}(0,2)$.

Our main result is a quantitative version of the Tracy--Widom laws for the largest eigenvalue of generalized Wigner matrices satisfying Assumption \ref{assumption_1}.

\begin{theorem}\label{thm_main_result}
	Let $H$ be a real symmetric or complex Hermitian generalized Wigner matrix satisfying Assumption \ref{assumption_1} and denote its largest eigenvalue by $\lambda_N$. For any fixed $r_0 \in \R$ and any small $\omega>0$, 
	\begin{equation} \label{result}
		\sup_{ r > r_0 }\Big|\P \Big( N^{2/3} (\lambda_N-2) < r \Big) - \mathrm{TW}_{\beta}(r) \Big| \leq  N^{-\frac{1}{3}+\omega},
	\end{equation}
	for sufficiently large $N \geq N_0(r_0,\omega)$. The corresponding statement holds true for the smallest eigenvalue.
\end{theorem}

The proof of Theorem \ref{thm_main_result} is based on the Green function comparison method initiated by Erd\H{o}s, Yau and Yin \cite{rigidity} to prove non-quantitative Tracy--Widom laws for Wigner matrices. Let
\begin{equation}\label{Green_fun}
	G(z):=\frac{1}{ H-z}\,, \qquad m_N(z):=\frac{1}{N} \Tr G(z)\,, \quad\qquad z =E+\ii \eta\in \C^+\,,
\end{equation}
denote the {\it resolvent} or {\it Green function} of the matrix $H$ and $m_N$ its normalized trace. The distribution of the rescaled largest eigenvalue can be linked to the expectation (of smooth functions) of $\Im m_N (z)$  for appropriately chosen spectral parameters $z$; see Lemma \ref{old_lemma} below.  Hence Theorem \ref{thm_main_result} follows from the Green function comparison in Theorem \ref{green_comparison} with $\eta$ chosen slightly above $N^{-1}$. Before we give the formal statement, we recall the local law for the Green function, which is a key tool in this paper.

\subsection{Local law for the Green function}
For a probability measure $\nu$ on $\R$, denote by $m_\nu$ its Stieltjes transform, i.e.,
\begin{align}
	m_\nu(z)\deq\int_\R\frac{\dd\nu(x)}{x-z}\,,\qquad z=E+\ii \eta\in\C^+\,.
\end{align}
Note that $m_{\nu}\,:\C^+\rightarrow\C^+$ is analytic and can be analytically continued to the real line outside the support of $\nu$. Moreover, $m_{\nu}$ satisfies $\lim_{\eta\nearrow\infty}\ii\eta {m_{\nu}}(\ii \eta)=-1$. The Stieltjes transform of the semicircle law $\dd \mu_{sc}(x):=\rho_{sc}(x) \dd x=\frac{1}{2 \pi} \sqrt{4-x^2} \mathds{1}_{[-2,2]} \dd x$, denoted by $m_{sc}(z)$, is the unique analytic solution $\C^+ \rightarrow \C^+$ to the equation 
\begin{equation}\label{msc}
	m_{sc}^2(z)+zm_{sc}(z)+1=0.
\end{equation} 
The Stieltjes transform $m_{sc}$ has the following quantitative properties, for a reference, see \eg~\cite{book}.
\begin{lemma}\label{prop_msc} The imaginary part of the Stieltjes transform of the semicircular law satisfies
		\begin{equation}\label{22}
			|{\Im}  m_{sc}(z)|  \sim \begin{cases}
				\sqrt{\kappa+\eta}, & \mbox{if } E \in[-2, 2], \\
				\frac{\eta}{\sqrt{\kappa+\eta}}, & \mbox{otherwise}\,,
			\end{cases}
		\end{equation}		
		uniformly in $z\in \{E+\ii \eta:|E| \leq 5, 0< \eta \leq 10 \}$, with $\kappa:=\min \{ |E-2|, |E+2| \}$. Moreover, $|m_{sc}(z)|\le 1$ holds on the same spectral domain.\end{lemma}

Before we state the local law for the Green function of $H$, we introduce the following spectral domain: For any given fixed $\epsilon>0$, let
\begin{equation}\label{ddd}
	\mathcal{S}= \mathcal{S}(\epsilon):=\big\{z=E+\ii \eta:  |E| \leq 5, N^{-1+\epsilon} \leq \eta \leq  10 \big\}.
\end{equation}

\begin{theorem}[Theorem 2.1 in \cite{rigidity}, Theorem 2.12 in \cite{Alex+Erdos+Knowles+Yau+Yin}]\label{locallawgene}
	Let $H$ be a generalized Wigner matrix satisfying Assumption \ref{assumption_1}. The following estimates hold uniformly in $z \in \mathcal{S}$,
	\begin{equation}\label{le law local}
		\max_{i,j} | G_{ij}(z) -\delta_{ij} m_{sc}(z) |   \prec \sqrt{ \frac{\Im m_{sc}(z)}{N \eta}} +\frac{1}{N \eta};\qquad | m_N(z) -m_{sc}(z) | \prec \frac{1}{N \eta}.
	\end{equation}	
\end{theorem}

As a corollary of Theorem \ref{locallawgene}, we have the following estimates on the eigenvalue rigidity and eigenvector delocalization.
Denote by $(\lambda_j)_{j=1}^N$ the eigenvalues of $H$ arranged in non-decreasing order. 
The corresponding eigenvectors are denoted by $(\mathbf{u_{j}} )_{j=1}^N$. For any $E_1<E_2$ ($E_1,E_2\in\R \cup\{\pm\infty\}$) define the eigenvalue counting function by
\begin{equation}\label{number_particle}
	\mathcal{N}(E_1,E_2):=\# \{j: E_1 \leq \lambda_j \leq E_2\}\,. 
\end{equation}  
We also define the classical location $\gamma_j$ of the $j$-th eigenvalue $\lambda_j$ by
\begin{equation}\label{classical}
	\frac{j}{N}=\int_{-\infty}^{\gamma_j}  \dd \mu_{\mathrm{sc}} (x).
\end{equation}

\begin{theorem}[Theorem 2.2 in \cite{rigidity}]\label{rigidity}
	For any $E_1<E_2$ and $1\leq j \leq N$, we have the rigidity estimate for the eigenvalues 
	\begin{equation}\label{rigidity_estimate}
		\Big| \mathcal{N}(E_1,E_2)-N \int_{E_1}^{E_2} \dd \mu_{\mathrm{sc}} (x)  \Big| \prec 1, \qquad |\lambda_j-\gamma_j| \prec N^{-2/3} \Big( \min \{ j, N-j+1\} \Big)^{-1/3}.
	\end{equation}
	For any deterministic unit vector $\mathbf{v} \in \C^{N}$, we have the delocalization estimate for the eigenvectors
	\begin{align}\label{delocal_vector}
		|\langle \mathbf{v}, \mathbf{u}_j\rangle|^2 \prec \frac{1}{N},
	\end{align}
	uniformly for any $1\leq j\le N$.
\end{theorem}

\subsection{Green function comparison theorem (GFT)}

With the help of the local law of the Green function in Theorem \ref{locallawgene} and the rigidity estimates of the eigenvalues in Theorem \ref{rigidity}, following~\cite{rigidity} we can link the distribution of the largest eigenvalue to a properly chosen observable in terms of the Green function as follows. 
\begin{lemma}[Lemma 2.5 in \cite{Schnelli+Xu}]\label{old_lemma}
Fix a small $\epsilon>0$ and a large $\Gamma>2/3$. We set 
$$E_L:=2+4N^{-2/3+\epsilon}.$$
For any $|E-2| \leq N^{-2/3+\epsilon}$ and $\eta \leq N^{-2/3-\epsilon}$, we have
\begin{align}\label{F_approx}
	\E \Big[ F\Big( N \int_{E-l}^{E_L} \Im m_N(y+\i \eta) \dd y \Big)\Big]-N^{-\Gamma} \leq \P\big( \lambda_N<E \big) \leq   \E \Big[F\Big( N \int_{E+l}^{E_L} \Im m_N (y+\i \eta) \dd y \Big)\Big]+N^{-\Gamma},
\end{align}
with $l=N^{6\epsilon} \eta$, and $F\,:\,\R\longrightarrow\R$ is a smooth cut-off function such that
\begin{equation}\label{q_function}
	F(x)=1, \quad \mbox{if} \quad |x| \leq 1/9; \qquad F(x)=0, \quad \mbox{if} \quad |x| \geq 2/9,
\end{equation}
and we assume that $F(x)$ is non-increasing for $x \geq 0$.
\end{lemma} 

To prove the quantitative Tracy--Widom laws in Theorem \ref{thm_main_result}, we establish the following GFT with an improved error term of size nearly $N^{-1/3}$.
	Compared to previous GFTs near the edges in \eg \cite{BEY edge universality, rigidity, sparse,Schnelli+Xu}, we remove the second moment matching restriction. In the meanwhile, the spectral resolution parameter $\eta$ is chosen down to slightly above~$N^{-1}$, much smaller than the typical eigenvalue spacing $N^{-2/3}$.

 \begin{theorem}[Green function comparison theorem near the edges]\label{green_comparison}
	Consider a generalized Wigner matrix $H$ satisfying Assumption \ref{assumption_1}. Let $F$ be a smooth function with uniformly bounded derivatives. Fix a small $\epsilon>0$, constants $C_1,C_2>0$, and choose $E_1,E_2$ and $\eta$ such that $2-C_1 N^{-2/3}\leq E_1<E_2\leq 2+C_2  N^{-2/3+\epsilon}$ and $N^{-1+\epsilon} \leq \eta \leq N^{-2/3-\epsilon}$.  Then for any small $\tau>0$, we have
	\begin{align}\label{green_difference}
		\Big| \big(\E-\E^{\mathrm{G \beta E}}\big)  \Big[ F\Big(N  \int_{E_1}^{E_2} \Im  m_N(x+\ii \eta) \dd x\Big) \Big]\Big| \leq  N^{-\frac{1}{3}+\tau},
	\end{align}
	for sufficiently large $N \geq N_0(C_1,C_2,\epsilon, \tau)$. The results hold true for both the real symmetric and complex Hermitian symmetry class. At the lower spectral edge the corresponding results also holds true.
\end{theorem}

Theorem~\ref{thm_main_result} is now obtained as follows. Using \eqref{F_approx} and~(\ref{green_difference}) together with the corresponding estimates for the Gaussian ensembles in \cite[Theorems 1 and 2]{gaussian_speed}, we obtain
\begin{align}
	\sup_{ r > r_0 }\Big|\P \Big( N^{2/3} (\lambda_N-2) < r \Big) - \mathrm{TW}_{\beta}(r) \Big| \leq	N^{-\frac{1}{3}+\tau}+CN^{\frac{2}{3}} l,
\end{align}
with $l=N^{6\epsilon} \eta$, where the first error term is from the Green function comparison in (\ref{green_difference}) and the second error term stems from the approximation for the distribution of the largest eigenvalue in (\ref{F_approx}). We hence prove Theorem \ref{thm_main_result} by choosing $\eta=N^{-1+\epsilon}$ with a small fixed $\epsilon<\omega/7$. Detailed arguments can be found in \cite[Theorem 1.3]{Schnelli+Xu} to prove quantitative Tracy--Widom laws for Wigner matrices; see also~\cite[Section~5]{Bourgade extreme}.

 Our convergence rate estimate thus depends on the approximation in~\eqref{F_approx} and the comparison estimate in~\eqref{green_difference} which both give comparable contributions. Whether the convergence rate can further be improved by including a spectral shift of the edge, similar to the GOE in~\eqref{gaussian}, remains an open question. Such a spectral shift would depend on the spectral properties of the variance matrix~$S$ and the fourth order cumulants of the matrix entries of $H$. Even for Gaussian matrices with inhomogeneous variance profiles the optimal convergence rate is unknown.

 We conclude this subsection with two comments on the assumed properties of the variance matrix $S$ in Assumption~\ref{assumption_1}. The summation condition in~\eqref{sum_S} insures that the limiting eigenvalue distribution is given by the semicircle law. For non-doubly stochastic variance matrices the limiting eigenvalue distribution is obtained via solutions to the (vector) Dyson equation~\cite{AEK17}. The edge universality for these type of models were obtained in~\cite{XXX}. We expect that the methods in the present paper are sufficiently robust to also derive quantitative Tracy--Widom laws for matrix models with such general variance profiles. Second, the lower bound in~\eqref{flat} insures a spectral gap of order one in~\eqref{S_spectrum}. Random band matrices are a prominent class of matrix models for which the spectral gap closes as $N$ increases. For sufficiently large bandwidths the fluctuations of the extremal eigenvalues are still given by the Tracy--Widom laws~\cite{So1}. Establishing convergence rates estimates for these models remains a challenging open problem.

\subsection{Cumulant expansions}

We end this section by recalling the following cumulant expansion formula which is another key tool of this paper, see \eg Lemma 3.1 in \cite{moment} for reference and Lemma 7.1 in there for the complex version.
\begin{lemma}\label{cumulant}
	Let $h$ be a real-valued random variable with finite moments. The $p$-th cumulant of $h$ is given by
	\begin{align}\label{cumulant_k}
		c^{(p)}(h):=(-\ii)^{p} \frac{\dd^p}{\dd t^p}\Big( \log \E \ee^{\ii t h} \Big)\Big|_{t=0}.
	\end{align}
	Let $f: \R \longrightarrow \C$ be a smooth function that has bounded derivatives and denote by $f^{(p)}$ its $p$-th derivative. Then for any fixed $l \in \N$, we have
	\begin{align}\label{le first cumulant formula_real}
		\E \big[h f(h)\big]=\sum_{p+1=1}^l \frac{1}{p!} c^{(p+1)}(h)\E[ f^{(p)}(h) ]+R_{l+1}\,,
	\end{align}
	where the error term satisfies
	\begin{align}\label{cumulant_error_real}
		|R_{l+1}| \leq C_l \E \big[ |h|^{l+1}\big] \sup_{|x| \leq M} |f^{(l)}(x)| +C_l \E \big[ |h|^{l+1} 1_{|h|>M}\big] \sup_{x \in \R} |f^{(l)}(x)|,
	\end{align}
	with $M>0$ being an arbitrary fixed cutoff.
\end{lemma}

 The usefulness of cumulant expansions in random matrix theory was recognized in~\cite{KKP} and has widely been used since, \eg~\cite{BoutetdeMonvel,ErdoesKruegerSchroeder,moment,He+Knowles,sparse,LP}. Iterative cumulant expansions due to unmatched indices~\cite{Erdos+Knowles+Yau,rigidity} were used in~\cite{He+Knowles,HLY,sparse} and systematically developed in~\cite{Schnelli+Xu,SXsample}.

\section{Special case of Theorem \ref{green_comparison}: $F(x)=x$ and proof strategy}\label{sec:toy}

Before we give the proof of the Green function comparison in Theorem \ref{green_comparison}, we consider the special case $F(x)=x$ and obtain the following proposition. The proof of Proposition \ref{GCT_mn} is then extended in Section~\ref{sec:proof} to prove Theorem \ref{green_comparison} for a general $F$. In fact, the improved estimate in (\ref{img_wigner}) below turns out to be a key input in the proof of Theorem~\ref{green_comparison} for a general $F$. 
\begin{proposition}\label{GCT_mn}
	Consider a generalized Wigner matrix satisfying Assumption \ref{assumption_1}.	Fix a small $\epsilon>0$ and constants $C_1,C_2>0$, define the domain of the spectral parameter $z$ near the upper edge as
	\begin{align}\label{S_edge}
		{\mathcal S}_{\mathrm{edge}} =& {\mathcal S}_{\mathrm{edge}}(\epsilon,C_1,C_2)\nonumber\\
		:=&\{ z=E+\ii \eta \in \mathcal{S}: -C_1 N^{-2/3} \leq E-2 \leq C_2  N^{-2/3+\epsilon}, N^{-1+\epsilon} \leq \eta \leq N^{-2/3-\epsilon}\}\,,
	\end{align}
	with $\mathcal{S}$ the spectral domain given in (\ref{ddd}). Then there exists $C>0$ depending on $C_1,C_2$ such that
	\begin{align}\label{img_wigner}
	 \E[\Im m_N(z)] \leq C N^{-1/3},
	\end{align}
	uniformly in $z \in {\mathcal S}_{\mathrm{edge}}(\epsilon,C_1,C_2)$, for sufficiently large $N \geq N_0(C_1,C_2,\epsilon)$.
\end{proposition}

We first claim the corresponding estimate for the Gaussian invariant ensembles relying on the explicit formula for the eigenvalue statistics~\cite{Mehta}.
\begin{lemma}\label{img_gaussian}
	Consider the Gaussian ensembles. Then there exists $C'>0$ depending on $C_1,C_2$ such that
		\begin{align}\label{img_gaussian_eq}
	\E^{\mathrm{G \beta E}}[\Im m_N(z)] \leq C' N^{-1/3}
\end{align}
uniformly in $z \in {\mathcal S}_{\mathrm{edge}}(\epsilon,C_1,C_2)$, for sufficiently large $N \geq N_0(C_1,C_2,\epsilon)$.
\end{lemma}
We remark that a slightly different estimate was obtained in \cite[Lemma 5.4]{Schnelli+Xu} for $z$ in a broader regime than (\ref{S_edge}). In fact, we can prove Lemma \ref{img_gaussian} using the uniform convergence for the correlation kernel \cite[Theorem 1.1]{Deift} of Gaussian ensembles and similar arguments in \cite[Lemma 6.1]{SXsample} applied for Laguerre ensembles. We omit the proof details for brevity.

Next, we  extend the estimate for the Gaussian ensembles in Lemma \ref{img_gaussian} to generalized Wigner matrices as stated in Proposition \ref{GCT_mn}. Before we give the formal proof of Proposition \ref{GCT_mn}, we first outline the proof strategy which will also be used to prove Theorem \ref{green_comparison} for a general $F$.

\subsection{Proof strategy}\label{sec:strategy}
We consider a generalized Wigner matrix with a given variance profile matrix $S$ satisfying the conditions in (\ref{sum_S}) and (\ref{flat}). To compare this matrix ensemble with the Gaussian invariant ensemble of the same symmetry class, we divide the comparison into two steps. 

 \subsubsection{Step one:} In the first step, we consider a generalized Wigner matrix with independent Gaussian entries, denoted by $W^{S}$, with the given variance profile matrix $S$.  We perform the comparison between the Gaussian matrix $W^{S}$ and the corresponding Gaussian invariant ensemble $\mathrm{G \beta E}$ which is independent of $W^{S}$, via the interpolating matrix flow
	$$	H^{(1)}(t)=\mathrm{e}^{-\frac{t}{2}}\mathrm{G \beta E} +\sqrt{1-\mathrm{e}^{-t}} W^{S}\,,\qquad  t \in \R^+.$$
	The corresponding Green function and its normalized trace are denoted by $G^{(1)}=G^{(1)}(t,z)$ and $m_N^{(1)}(t,z)$. Though the variances of the matrix entries  of $H^{(1)}(t)$ vary with time in general, the good news is that these matrix entries remain Gaussian distributed. Hence the higher (than the second) order cumulants vanish automatically. So this step is mainly to estimate the contributions to the time derivative of $\E[m_N^{(1)}(t,z)]$ from the inhomogeneous variances of the entries of $W^S$. More precisely, taking the time derivative of $\E[m_N^{(1)}(t,z)]$, we obtain 
	\begin{align}\label{strategy}
		\frac{\dd}{\dd t}\E[\Im m^{(1)}_N(t,z)]=\frac{\ee^{-t}}{N} \sum_{v,a,b}  \big(S_{ab}-\frac{1}{N}\big)\E \Big[\Im \big( G^{(1)}_{vb} G^{(1)}_{bv} G^{(1)}_{aa}+G^{(1)}_{va} G^{(1)}_{ab} G^{(1)}_{bv} \big)(t,z)\Big].
	\end{align}
Note that if $S_{ab} = N^{-1}$, then the right side will vanish trivially. With inhomogeneous variances, using the local law in Theorem \ref{locallawgene} and the {\it Ward identity}
\begin{align}\label{ward_id}
 \sum_{j=1}^N |G_{ij}|^2=\frac{1}{\eta} \Im G_{ii},
\end{align}
we can bound (\ref{strategy}) by
$$	\frac{\dd}{\dd t}\E[\Im m^{(1)}_N(t,z)] =  O_\prec\Big(\frac{\ee^{-t}\E[\Im m^{(1)}_N(t,z)]}{\eta}\Big),$$
which is far from the target estimate leading to (\ref{img_wigner}). Since the variances of the matrix entries are only of size $O(N^{-1})$, these second order terms in (\ref{strategy}) are considerably harder to estimate than the third and fourth order terms for non-Gaussian matrices, which will be considered in {\it Step two} below.

If the row and column indices of the Green function entries in (\ref{strategy}) were not related to the index $a$ (or~$b$) appearing in the coefficient $S_{ab}$, then these terms in (\ref{strategy}) would vanish trivially since $S$ is a doubly stochastic matrix. Hence, splitting $G^{(1)}_{aa}=m_{sc}+(G^{(1)}_{aa}-m_{sc})$ from the local law in~\eqref{le law local}, we get
\begin{align}\label{le step 1}
\frac{1}{N} \sum_{v,a,b}  \big(S_{ab}-\frac{1}{N}\big)\E\Big[\Im \big( G^{(1)}_{vb} G^{(1)}_{bv} G^{(1)}_{aa}\big)\Big]&=
\frac{1}{N} \sum_{v,a,b}  \big(S_{ab}-\frac{1}{N}\big)\Im \E\Big[m_{sc}G^{(1)}_{vb} G^{(1)}_{bv}+G^{(1)}_{vb} G^{(1)}_{bv} \big( G^{(1)}_{aa}-m_{sc}\big)\Big]\nonumber\\
&=
\frac{1}{N}\sum_{v,a,b}  \big(S_{ab}-\frac{1}{N}\big)\Im \E\Big[G^{(1)}_{vb} G^{(1)}_{bv} \big( G^{(1)}_{aa}-m_{sc}\big)\Big]\,.
\end{align}

Next, we can further decouple the index $a$ on the right side of~\eqref{le step 1} using cumulant expansions. In order to study such terms in general, we introduce in Subsection \ref{subsec:abstract_form} an abstract expansion mechanism for any term of the form in (\ref{form}) below, a product of the entries of the Green function $G$ and the entries of the variance profile matrix $S$. When we apply the cumulant expansions to such a term by expanding a Green function entry with an index $a$, we obtain from the inhomogeneous variances an additional coefficient factor $S_{ak}$ with $k$ being a fresh summation index, see \eg (\ref{expand_index_a_step2}) below. As the result, the leading term is then given by a product of Green function entries that are free of the index $a$ (\ie with $a$ being replaced by the fresh index $k$) with the coefficient $S_{ab}S_{ak}$. Summing over $a$, the coefficient $S_{ab}S_{ak}$ is then given by $(S^2)_{kb}$. For example, we have (ignoring some irrelevant factor)
\begin{align}\label{strategy_term}
	&\frac{1}{N} \sum_{v,a,b}  \big(S_{ab}-\frac{1}{N}\big)\Im\E \Big[ G^{(1)}_{vb} G^{(1)}_{bv} \big(G^{(1)}_{aa}-m_{sc}\big)\Big]\nonumber\\
	&=\frac{1}{N} \sum_{v,b,k}  \big((S^2)_{kb}-\frac{1}{N}\big)\Im\E \Big[ G^{(1)}_{vb} G^{(1)}_{bv} \big(G^{(1)}_{kk}-m_{sc}\big)\Big]+\mbox{sub-leading terms}.
\end{align}
Note that the fresh index $k$ plays the same role as the original index $a$ and we can further expand this leading term to get even higher powers of $S$. Using the spectral property of $S$ in (\ref{S_spectrum}), $(S^k)_{ab}$ tends to be very close to $N^{-1}$ with sufficiently large exponents $k \sim \log N$. Hence the leading term in~(\ref{strategy_term}) will almost vanish and we end up with $O(\log N)$ subleading terms consisting of more off-diagonal Green function entries. Iteratively expanding these subleading terms in combination with the generalized Ward identity in~(\ref{ward_1})-(\ref{ward_2}) below, the second order terms in~(\ref{strategy}) can be bounded effectively using $\E[\Im m^{(1)}_N(t,z)]$, \ie
	\begin{align}\label{green_difference_2_step3}
	\Big|\frac{\dd \E[\Im m^{(1)}_N(t,z)]}{\dd t}\Big| \leq \ee^{-t} \Big( o(1) 
	\E[\Im m^{(1)}_N(t,z)]+ o\big(N^{-1/3}\big) \Big).
\end{align}

	Then using Gr\"onwall's inequality, we extend the estimate in (\ref{img_gaussian_eq}) for the initial Gaussian invariant ensemble to the terminal matrix ensemble $W^{S}$, \ie for any $t\in \R^+$,
	\begin{align}\label{green_difference_2_step2}
		\E[\Im m^{(1)}_N(t,z)] \leq \E[\Im m^{(1)}_N(0,z)] \exp\Big( o(1) \int_{0}^t  \ee^{-s}  \dd s \Big)+o\big(N^{-\frac{1}{3}}\big)=O(N^{-1/3}).
	\end{align}
We remark that getting the $o(1)$ factor in front of $\E[\Im m^{(1)}_N(t,z)]$ in (\ref{green_difference_2_step3}) is essential to apply Gr\"onwall's inequality in (\ref{green_difference_2_step2}).

\subsubsection{Step two:}	In the second step, we compare the Gaussian matrix~$W^{S}$ considered in {\it Step one} with a generic random matrix $H$ with the same variance profile matrix $S$, via the interpolating matrix flow
	$$	H^{(2)}(t)=\mathrm{e}^{-\frac{t}{2}} W^{S} +\sqrt{1-\mathrm{e}^{-t}}H \,,\qquad  t \in \R^+.$$
	Since the first two moments of the matrix entries of both $W^{S}$ and $H$ are the same,
	 this step is to estimate the contributions to the time derivative of $ \E[\Im m^{(2)}_N(t,z)]$ from higher order (\ie  third and fourth order) moments of the matrix entries of $H$; see (\ref{sde_im}) below. The proof of this step is in the same spirit of \cite[Theorem 1.4]{Schnelli+Xu} for Wigner matrices. The only difference is that the variances of matrix entries are no longer identical and we need to extend the arguments in \cite{Schnelli+Xu} to the inhomogeneous cases.

	More precisely, for the third order terms in (\ref{sde_im}) with the so-called unmatched indices (see Definition~\ref{unmatch_def}), we adapt the expansion mechanism in \cite[Section 6]{Schnelli+Xu} to generalized Wigner matrices with inhomogeneous variances. By performing the expansions iteratively, we show that these third order terms with unmatched indices can be bounded by $O(N^{-1/2})$. Moreover, using the generalized Ward identity in~(\ref{ward_1})-(\ref{ward_2}), the remaining fourth order terms in (\ref{sde_im}) can be bounded by $o(1)\E[\Im m^{(2)}_N(t,z)]$ effectively. Therefore, we obtain
		\begin{align}\label{green_difference_2_step4}
		\Big|\frac{\dd \E[\Im m^{(2)}_N(t,z)]}{\dd t}\Big| \leq \ee^{-t} \Big( o(1) 
		\E[\Im m^{(1)}_N(t,z)]+ O\big(N^{-1/2}\big) \Big).
	\end{align}
	Then using Gr\"onwall's inequality as in (\ref{green_difference_2_step2}), we can extend the estimate for the initial Gaussian matrix~$W^{S}$ obtained in (\ref{green_difference_2_step2}) to the generalized Wigner matrix~$H$. This Gr\"onwall argument shortens our proof compared to the recursive comparison arguments used in \cite[Section 5]{Schnelli+Xu}.

\subsection{Proof of Proposition \ref{GCT_mn}}\label{sec:proof_prop}
To present the proof, we will consider only real symmetric generalized Wigner matrices for notational simplicity, though the real cases are theoretically heavier than the complex cases.

\begin{proof}[Proof of Proposition \ref{GCT_mn}]

 As explained in Section \ref{sec:strategy}, the proof is divided into two steps. 
 
 In the first step, we perform the comparison between the GOE matrix estimated in Lemma \ref{img_gaussian} and an independent generalized Wigner matrix $W^{S}$ with Gaussian entries and a given variance profile matrix $S$ satisfying (\ref{sum_S}) and (\ref{flat}). To accommodate the $\frac{1}{1+\delta_{ab}}$ factor in the differentiation rule in (\ref{dH}) below for the real case, we instead replace $W^{S}$ with a slightly different Gaussian matrix ${W}^{\wt S}$ whose variance profile matrix is given by
 \begin{align}\label{tilde_S}
 	(\wt S)_{ab} =(S)_{ab} (1+\delta_{ab}).
 \end{align}
We remark that in the complex case this step is not necessary. We claim that such modifications on the variances of diagonal entries does not influence the statement in Proposition \ref{GCT_mn}.
 \begin{lemma}\label{lemma_third_step}
 	We will use $\E^{W^{S}}$ and $\E^{W^{\wt S}}$ to denoted the corresponding expectation for the Gaussian matrix with variance profile matrix $S$ and $\wt S$ respectively. Then $\E^{W^{\wt S}}[\Im m_N(z)]=O(N^{-1/3})$ will imply 
 	\begin{align}\label{will_prove}
 		\big| \E^{W^{ S}}[\Im m_N(z)] \big| =O(N^{-1/3}),
 	\end{align}
 and vice versa.
 \end{lemma}
The proof of Lemma \ref{lemma_third_step} is postponed to the appendix. We then consider the following modified matrix interpolating flow
\begin{equation}\label{sum_2}
	H^{(1)}(t):=\mathrm{e}^{-\frac{t}{2}}W +\sqrt{1-\mathrm{e}^{-t}} W^{\wt S}\,,\qquad  t \in \R^+,
\end{equation}
where $W$ is the standard GOE matrix which is independent of $W^{\wt S}$. The Green function of $H^{(1)}(t)$ and its normalized trace are denoted by $G^{(1)} = G^{(1)}(t,z)$ and $m_N^{(1)}(t,z)$. We will use a Gr\"onwall argument in Section \ref{sec:step1} to show that $\E[\Im m^{(1)}_N(t, z)]$ has a similar upper bound as the initial Gaussian invariant ensemble in (\ref{img_gaussian_eq}).

\begin{proposition}\label{GCT_mn_1}
There exists $C>0$ depending on $C_1,C_2$ such that
	\begin{align}\label{img_wigner_1}
		\E[\Im m^{(1)}_N(t, z)] \leq C N^{-1/3}.
	\end{align}
	uniformly in $z \in {\mathcal S}_{\mathrm{edge}}(\epsilon,C_1,C_2)$ in (\ref{S_edge}) and $t \geq 0$, for sufficiently large $N \geq N'_0(C_1,C_2,\epsilon)$. In particular, we have $\E^{W^{\wt S}}[\Im m_N(z)]=O(N^{-1/3})$, which will further imply that
	\begin{align}\label{middle}
		\E^{W^{S}}[\Im m_N(z)]=O(N^{-1/3}).
	\end{align}
\end{proposition}

In the second step, we perform the comparison between the generalized Gaussian matrix $W^{S}$ with the original variance profile $S$ and an independent generalized Wigner matrix $H$ with the same variance profile via the interpolating flow
\begin{equation}\label{sum_1}
	H^{(2)}(t):=\mathrm{e}^{-\frac{t}{2}}W^{S}  +\sqrt{1-\mathrm{e}^{-t}}H \,,\qquad  t \in \R^+.
\end{equation}
Using a Gr\"onwall argument, we will show in Section \ref{sec:step2} that $\E[\Im m^{(2)}_N(t, z)]$ has a similar upper bound as the initial generalized Gaussian matrix $W^S$ estimated in (\ref{middle}).
\begin{proposition}\label{GCT_mn_2}
	There exists $C'>0$ depending on $C_1,C_2$ such that
	\begin{align}\label{img_wigner_2}
		\E[\Im m^{(2)}_N(t, z)] \leq C' N^{-1/3},
	\end{align}
	uniformly in $z \in {\mathcal S}_{\mathrm{edge}}(\epsilon,C_1,C_2)$ in (\ref{S_edge}) and $t \geq 0$, for sufficiently large $N \geq N'_0(C_1,C_2,\epsilon)$. In particular, for the generalized Wigner matrix $H$ in (\ref{sum_1}), we have
		\begin{align}
		\E^{H}[\Im m_N(z)]=O(N^{-1/3}).
	\end{align}
\end{proposition} 
 Hence we finish the proof of Proposition \ref{GCT_mn}
\end{proof}

\section{From GOE to Gaussian matrix with a variance profile: Proof of Proposition \ref{GCT_mn_1}}\label{sec:step1}
Given the matrix flow $H^{(1)}(t)$ in (\ref{sum_2}), in order to prove Proposition \ref{GCT_mn_1}, it suffices to prove the following lemma on the time derivative of $\E[m_N^{(1)}(t,z)]$.
\begin{proposition}\label{lemma_step_1}
	For any $t\geq 0$ and $z \in {\mathcal S}_{\mathrm{edge}}$, we have
	$$\Big|\frac{\dd \E[\Im m^{(1)}_N(t,z)]}{\dd t}\Big| \prec  \ee^{-t}\Big(
	N^{-\epsilon/4}\E[\Im m^{(1)}_N(t,z)]+N^{-1/3-\epsilon/4}\Big).$$
\end{proposition}
Admitting Proposition \ref{lemma_step_1}, using Gr\"onwall's inequality in combination with the corresponding estimate in Lemma \ref{img_gaussian} for the initial GOE matrix, we have
\begin{align}\label{conclusion_1}
	\E[\Im m^{(1)}_N(t,z)] \leq C \E[\Im m^{(1)}_N(0,z)]+O_{\prec}(N^{-1/3-\epsilon/4})=O(N^{-1/3}),
\end{align}
uniformly in $z \in {\mathcal S}_{\mathrm{edge}}$ and $t \in \R^+$.  Hence we have proved the first statement (\ref{img_wigner_1}) of Proposition \ref{GCT_mn_1}. 

Choose $T_0:=10\log N$. In view of the matrix flow in (\ref{sum_2}), one shows that $G^{(1)}(T_0,z)$ is very close to $G^{(1)}(\infty,z):=(W^{\wt S}-z)^{-1}$, \ie
\begin{align}\label{approxxxx}
	\|G^{(1)}(T_0,z) -G^{(1)}(\infty,z)\|_{\mathrm{max}} \leq& \|G^{(1)}(T_0,z)(W^{\wt S}-H^{(1)}(T_0))G^{(1)}(\infty,z)\|\nonumber\\
	\leq& \frac{N}{\eta^2} \|W^{\wt S}-H^{(1)}(T_0))\|_{\mathrm{max}} \prec N^{-2},
\end{align}
where we used the inequality $\|A\|_{\mathrm{max}} \leq \|A\| \leq N \|A\|_{\mathrm{max}}$, the second resolvent identity and  that $\|G(z)\| \leq \frac{1}{\eta}$. 
Combining (\ref{conclusion_1}) with (\ref{approxxxx}), we have
$$\E^{\wt S}[\Im m_N(z)]=O(N^{-1/3}).$$
Thus the second statement (\ref{middle}) of Proposition \ref{GCT_mn_1} follows directly from Lemma \ref{lemma_third_step}.

The remaining part of this section is devoted to proving Proposition \ref{lemma_step_1}. Throughout the remaining part of this section, we often ignore the superscript $(1)$ and the dependence on $t \in \R^+, z \in \C^+$ and set
\begin{equation}\label{le time dependent G_1}
	H(t)= H^{(1)}(t); \quad	G = G^{(1)}(t,z)=\frac{1}{ H^{(1)}(t)-z}; \quad  m_N = m^{(1)}_N(t,z)=\frac{1}{N} \Tr G^{(1)}(t,z).
\end{equation}
Since $H(t)$ is real symmetric, the corresponding Green function satisfies
\begin{align}\label{symmetric}
	G_{ij}=G_{ji}, \qquad\qquad  1\leq i<j\leq N.
\end{align}
From the local law in Theorem \ref{locallawgene} and Lemma \ref{prop_msc}, we obtain a similar local law for $G(t,z)$, \ie
\begin{equation}\label{G}
	\max_{i,j} | G_{ij}(t,z) -\delta_{ij}  m_{sc} (z) |   \prec \Psi  :=\frac{1}{N\eta}, \qquad  N^{-\frac{1}{3}+\epsilon} \leq \Psi\leq N^{-\epsilon},
\end{equation}
uniformly in $t\in \R^+$ and $z \in {\mathcal S}_{\mathrm{edge}}$ given in (\ref{S_edge}). Indeed, for any $t\in [0, 10 \log N]$, the above time-dependent local law follows from Theorem \ref{locallawgene} and that $G(t,z)$ is stochastically Lipschitz continuous with time using a grid argument. For any large $t\geq 10 \log N$, the time-dependent local law can be obtained using a standard matrix perturbation theory. The estimates throughout this section hold uniformly for any $t\in \R^+, z \in {\mathcal S}_{\mathrm{edge}}$ without specific mentioning. 

Using the spectral decomposition of $H(t)$ and the analogous eigenvector delocalization estimates in (\ref{delocal_vector}), we have the following estimate for the diagonal Green function entries, \ie for any $1\leq a\leq N$,
\begin{align}\label{ward_1}
	\Im G_{aa}(t,z)=\sum_{j=1}^N \frac{ \eta |\langle \mathbf{e}_a, \mathbf{u}_j(t) \rangle|^2 }{|\lambda_j(t)-z|^2} \prec \frac{\eta}{N}\sum_{j=1}^N \frac{ 1}{|\lambda_j(t)-z|^2} =\Im m_N(t,z),
\end{align}
where $\{\lambda_j(t)\}$ are the eigenvalues of $H(t)$ and $\{\mathbf{u}_j(t)\}$ are the corresponding eigenvectors. As discussed below (\ref{G}), the estimate in (\ref{ward_1}) holds true for any $t\in \R^+$ and $z \in \mathcal{S}$ given in (\ref{ddd}). Using Young's inequality, the Ward identity in (\ref{ward_id}) and the estimate in (\ref{ward_1}),  
for any $1\leq a,c\leq N$, we have the following {\it generalized Ward identity}, 
\begin{align}\label{ward_2}
	\frac{1}{N} \sum_{b=1}^N |G_{ab}G_{bc}(t,z)|^2\prec \frac{\Im m_N(t,z)}{N\eta}.
\end{align}
We also remark that, since $\|G(z)\| \leq \frac{1}{\eta}$ which implies that $ \max_{i,j} |G_{ij}(z)|\leq N^{1-\epsilon}$ for any $z \in {\mathcal S}$, the condition in the third statement in Lemma \ref{dominant} is always satisfied to estimate the expectations of finite products of Green function entries without specific mentioning.

\subsection{Proof of Proposition \ref{lemma_step_1}}
Now we are ready to prove Proposition \ref{lemma_step_1}.

\begin{proof}[Proof of Proposition \ref{lemma_step_1}]
	Recall the matrix flow $H^{(1)}(t)= (h_{ab})_{1\leq a,b\leq N}$ in (\ref{sum_2}), \ie
	\begin{align}\label{sum_22}
		H^{(1)}(t)=\mathrm{e}^{-\frac{t}{2}}W +\sqrt{1-\mathrm{e}^{-t}} W^{\wt S}, \qquad (\wt S)_{ab} =(S)_{ab} (1+\delta_{ab}),
	\end{align}
	where $W=(w_{ab})_{1\leq a,b\leq N}$ is the GOE matrix and $W^{\wt S}=(w^{(\wt s)}_{ab})_{1\leq a,b\leq N}$ is an independent real symmetric Gaussian matrix with the modified variance profile $\wt S$. 
	
	Taking the time derivative of the expectation of $\Im m^{(1)}_N(t,z)$, we have
\begin{align}\label{intial_step}
	\frac{\dd}{\dd t}\E[\Im m^{(1)}_N(t,z)]=\frac{1}{N}\sum_{v=1}^N\E\Big[\Im \frac{\dd G_{vv}}{\dd t} \Big]=-\frac{1}{N}\sum_{v=1}^N\sum_{a, b=1}^{N} \E\Big[\Im \big( G_{va} G_{bv}\frac{\dd h_{ab}}{\dd t} \big)\Big],
\end{align}
where in the first step we interchanged the derivative and the expectation due to that $\|G(z)\| \leq \frac{1}{\eta}$, and in the second step we used (\ref{symmetric}) and the following differentiation rule for the Green function entries
\begin{align}\label{dH}
	\frac{\partial G_{ij}}{\partial h_{ab}}=-\frac{G_{ia}G_{bj}+G_{ib}G_{aj}}{1+\delta_{ab}}.
\end{align}
From (\ref{sum_22}), we have
\begin{align}\label{compute_deri}
	\frac{\dd h_{ab}}{\dd t}=-\frac{\ee^{-\frac{t}{2}}}{2} w_{ab}+\frac{\ee^{-t}}{2\sqrt{1-\ee^{-t}}} w^{(\wt s)}_{ab}.
\end{align}
Note that the $k$-th cumulants $(k\geq 2)$ of these Gaussian random variables are given by
\begin{align}\label{cumulant_exact}
       c^{(2)}(w_{ab}) = \frac{1+\delta_{ab}}{N},\quad c^{(2)}(w^{(\wt s)}_{ab}) =S_{ab}(1+\delta_{ab}),    \quad 	c^{(k)}(w_{ab}) =  c^{(k)}(w^{(\wt s)}_{ab}) = 0, \qquad  k\geq 3.
\end{align} 
Plugging (\ref{compute_deri}) in (\ref{intial_step}), using the cumulant expansion formula in Lemma \ref{cumulant} on $\{w_{ab}\}$ and $\{w^{(\wt s)}_{ab}\}$ respectively, we have
\begin{align}\label{sde_im_2}
	\frac{\dd}{\dd t}\E[\Im m^{(1)}_N(t,z)]=&-\frac{\ee^{-t}}{2} \frac{1}{N}\sum_{v,a,b=1}^N \big( S_{ab}-\frac{1}{N}\big) (1+\delta_{ab})\E \Big[ \frac{\partial \Im(G_{va}G_{bv})}{\partial h_{ab}}\Big]\nonumber\\
	=&\frac{\ee^{-t}}{N} \sum_{v,a,b}  T_{ab}\E \Big[\Im \big( G_{vb} G_{bv} G_{aa}+G_{va} G_{ab} G_{bv} \big)\Big],
\end{align}
where we defined a new matrix $T=(T_{ab})_{1\leq a,b\leq N}$
\begin{align}\label{T_ab}
	T:=S-\Pi,\qquad \qquad   (\Pi)_{ab} := \frac{1}{N}, \qquad 1\leq a,b\leq N.
\end{align} 
Here we used the chain rule and (\ref{cumulant_exact}) in the first step of (\ref{sde_im_2}), and in the second step we used the differentiation rule in (\ref{dH}), (\ref{symmetric}) and that $\partial/\partial h_{ab}$ commutes with $\Im$.

 It is easy to check that $T$ is a real symmetric matrix and commutes with $S$. Moreover, for any $k\geq 1$
 \begin{align}\label{relation_TS}
 	T^{k}=S^{k}-\Pi, \qquad \qquad  TS^{k}=T^{k+1}.
 \end{align}
From the conditions in (\ref{sum_S}), (\ref{flat}) and the spectral property of the matrix $S$ in (\ref{S_spectrum}), we have
\begin{lemma}\label{lemma_prop_T}
	For any $k\geq 1$, we have the following properties
	\begin{align}\label{prop_T}
		\sum_{b=1}^N(T^{k})_{ab} = 0, \qquad 	\|T^{k}\|_{\max} \leq \frac{C_0}{N},
	\end{align}
	where $C_0=2(C_{\mathrm{sup}}+1)$ with $C_{\mathrm{sup}} \geq C_{\mathrm{inf}}>0$ given in (\ref{flat}). Furthermore, there exists a constant $c_0$ with $C_{\mathrm{inf}}\leq c_0 \leq 1$ such that 
	\begin{align}\label{high_power_T}
	 \|T^{k}\|_{\max}\leq \|T^{k}\| \leq (1-c_0)^{k}.
	\end{align} 
\end{lemma}

Using the summation property of the matrix $T$ in (\ref{prop_T}), we then write (\ref{sde_im_2}) as
\begin{align}\label{goal}
	\frac{\dd}{\dd t}\E[\Im m^{(1)}_N(t,z)]=&\ee^{-t}\frac{1}{N} \sum_{v,a,b}  T_{ab}\E \Big[\Im \big( G_{vb} G_{bv} (G_{aa}-m_{sc})+G_{va} G_{ab} G_{bv} \big)\Big].
\end{align}
Note that if we replace the row and column index $a$ (or the index $b$) in the product of Green function entries in (\ref{goal}) with a fresh summation index, say $k$, then the resulting term will vanish due to the summation property of $T_{ab}$ in (\ref{prop_T}). Such replacements can be realized using cumulant expansions. 

For example, we look at the first term on the left side of (\ref{goal}). We ignore the imaginary part since it does not play an essential role here. Let
\begin{align}\label{example_step_0}
I_1:=\frac{1}{N} \sum_{v,a,b}  T_{ab}\E \big[ G_{vb} G_{bv} (G_{aa}-m_{sc}) \big]=O_\prec(N\Psi^3),
\end{align}
where the estimate follows naively from the local law in (\ref{G}). We will next perform cumulant expansions in Lemma \ref{cumulant} via the index $a$ to improve the naive estimate to
\begin{align}\label{small_goal}
I_1=O_\prec((\log N)N\Psi^4).
\end{align}
The corresponding expansions via the index $b$ are similar but slightly more complicated. 

Using the relation $-\frac{1}{m_{sc}}=z+m_{sc}$ (see (\ref{msc})), the resolvent identity $zG_{aa}=(HG)_{aa}-1$ and invoking cumulant expansions, we have
\begin{align}\label{example_step_1}
	-\frac{1}{m_{sc}} I_1=&\frac{1}{N} \sum_{v,a,b}  T_{ab}\E \Big[ G_{vb} G_{bv}\big( \sum_{k} h_{ak} G_{ka} \big)\Big]  +\frac{m_{sc}}{N} \sum_{v,a,b}  T_{ab}\E \big[ G_{vb} G_{bv} (G_{aa}-m_{sc}) \big]\nonumber\\
	=&\frac{1}{N} \sum_{v,a,b,k}  T_{ab} \times c_{ak}^{(2)}(t)   \E \big[  \frac{\partial G_{ka} G_{vb} G_{bv}  }{\partial h_{ak}}\big]+\frac{m_{sc}}{N} \sum_{v,a,b}  T_{ab}\E \big[ G_{vb} G_{bv} (G_{aa}-m_{sc}) \big],
\end{align}
where $c_{ak}^{(2)}(t)$ is the variance of the time dependent Gaussian entry $h_{ak}$, \ie
$$c_{ak}^{(2)}(t)=c^{(2)}(h_{ak})=\Big(\frac{\ee^{-t}}{N}+(1-\ee^{-t}) S_{ak}\Big)(1+\delta_{ak})$$ 
which follows from (\ref{cumulant_exact}). In combination with the $\frac{1}{1+\delta_{ak}}$ factor in the differentiation rule in (\ref{dH}), we define a real symmetric time-dependent matrix, \ie
\begin{align}\label{S_ab}
	(S(t))_{ak}	:=\frac{c_{ak}^{(2)}(t)}{1+\delta_{ak}}=\frac{\ee^{-t}}{N}+(1-\ee^{-t}) S_{ak}, \qquad 1\leq a,k\leq N.
\end{align}
It is easy to check that $S(t)$ is a symmetric and doubly stochastic matrix. More precisely, we have
\begin{align}\label{TS_prop}
	S(t)= (1-\ee^{-t})T+\Pi, \qquad 	S(t)T=TS(t)=(1-\ee^{-t}) T^{2}.
\end{align}
 Using the differentiation rule in (\ref{dH}), we have from (\ref{example_step_1}) that
\begin{align}\label{example_step_2}
	-\frac{1}{m_{sc}}I_1=&-\frac{1}{N} \sum_{v,a,b,k}  T_{ab} (S(t))_{ka}  \E \big[  G_{aa}G_{kk}G_{vb} G_{bv} \big]+\frac{m_{sc}}{N} \sum_{v,a,b}  T_{ab}\E \big[ G_{vb} G_{bv} (G_{aa}-m_{sc}) \big]\nonumber\\
	&-\frac{1}{N} \sum_{v,a,b,k}  T_{ab} (S(t))_{ka}  \E \big[ G_{ka}G_{ka}G_{vb} G_{bv}+G_{ka}G_{va}G_{kb} G_{bv} +G_{ka}G_{vk}G_{ab} G_{bv}\nonumber\\
	&\qquad \qquad\qquad\qquad\qquad+G_{ka}G_{vb}G_{ba} G_{kv}+G_{ka}G_{vb}G_{bk} G_{av}\big],
\end{align}
where the leading term on the right side is obtained from acting $\partial/\partial h_{ak}$ on $G_{ka}$, and the remaining terms from differentiating $\partial/\partial h_{ak}$ are presented on the last two lines. Each term on the last two lines of (\ref{example_step_2}) contains four off-diagonal Green function entries which can be bounded by $O_{\prec}(N \Psi^4)$ using the local law in (\ref{G}). Since $\sum_{k}(S(t))_{ka}= 1$ from (\ref{TS_prop}), we observe a cancellation on the first line of (\ref{example_step_2}), \ie 
\begin{align}\label{example_step_3}
\mbox{first line of the r.h.s. (\ref{example_step_2})}=&-\frac{1}{N} \sum_{v,a,b,k}  T_{ab} (S(t))_{ka}  \E \big[  G_{aa}(G_{kk}-m_{sc})G_{vb} G_{bv}\big]\nonumber\\
	=&- \frac{m_{sc}}{N}  \sum_{v,a,b,k}  T_{ab} (S(t))_{ka}  \E \big[  (G_{kk}-m_{sc})G_{vb} G_{bv}\big] \nonumber\\
	&-\frac{1}{N} \sum_{v,a,b,k}  T_{ab} (S(t))_{ka}  \E \big[  (G_{aa}-m_{sc})(G_{kk}-m_{sc})G_{vb} G_{bv}\big],
\end{align}
where the term on the last line can be bounded by $O_{\prec}(N\Psi^4)$ using the local law in (\ref{G}), while the leading term can only be bounded by $O_{\prec}(N\Psi^3)$ without any improvement. However for such leading term in (\ref{example_step_3}), we observe that the index $a$ no longer shows up in the Green function entries as the row or the column index. Using the second relation in (\ref{TS_prop}), we rewrite this leading term as
\begin{align}\label{leading_term}
	- \frac{m_{sc}}{N}  \sum_{v,a,b,k}  T_{ab} (S(t))_{ka}  \E \big[  (G_{kk}-m_{sc})G_{vb} &G_{bv}\big]=- \frac{m_{sc}}{N}  \sum_{v,b,k}   (S(t)T)_{kb}  \E \big[  (G_{kk}-m_{sc})G_{vb} G_{bv}\big]\nonumber\\
	=&- \frac{m_{sc}(1-\ee^{-t})}{N}  \sum_{v,b,a}   (T^2)_{ab}  \E \big[  (G_{aa}-m_{sc})G_{vb} G_{bv}\big],
\end{align}
where we replaced the summation index $k$ with the original index $a$ without loss of generality. Combining (\ref{example_step_3}), (\ref{leading_term}) with (\ref{example_step_2}) and multiplying $-m_{sc}$ on both sides of (\ref{example_step_2}), we conclude that
\begin{align}\label{example_step_4}
		 I_1=&\frac{m^2_{sc}(1-\ee^{-t})}{N}  \sum_{v,a,b}   (T^2)_{ab}  \E \big[  G_{vb} G_{bv}(G_{aa}-m_{sc})\big]+O_{\prec}(N\Psi^4),
	 \end{align}
where the last error term is from the terms on the last two lines of (\ref{example_step_2}) and the last line of (\ref{example_step_3}) using the local law in (\ref{G}). Note that the leading term on the right side of (\ref{example_step_4}) is the same as the original term $I_1$ in (\ref{example_step_0}) with the power of $T$ increased to $2$ up to a deterministic factor $m^2_{sc}(1-\ee^{-t})$. For this leading term we can repeat the above expanding procedure to further increase the power of $T$, \ie for any $k\geq 1$,
\begin{align}\label{example_step_5}
		\frac{1}{N} &\sum_{v,a,b}  (T^k)_{ab}\E \big[ G_{vb} G_{bv} (G_{aa}-m_{sc}) \big]\nonumber\\
		=&\frac{m^{2}_{sc}(1-\ee^{-t})}{N}  \sum_{v,a,b}   (T^{k+1})_{ab}  \E \big[  G_{vb} G_{bv}(G_{aa}-m_{sc})\big]+O_{\prec}(N\Psi^4).
 \end{align}
Note that the deterministic factor $m^2_{sc}(1-\ee^{-t})$ is not harmful in iterations since $|m_{sc}| \leq 1$ and $t\geq 0$. We stop the iterations until the power of $T$ is increased to sufficiently large, say
\begin{align}\label{K_choice}
	 K:=\lceil-\frac{10}{\log(1-c_0)} \log N\rceil,
\end{align}
with the constant $0<c_0<1$ given in Lemma \ref{lemma_prop_T}. From (\ref{high_power_T}), we have
$$\max_{a,b} \big|(T^K)_{ab} \big| \leq  (1-c_0)^{K} \leq N^{-10}.$$
Thus we obtain that
\begin{align}
	\Big|\frac{1}{N}  \sum_{v,a,b}   (T^K)_{ab}  \E \big[  G_{vb} G_{bv}(G_{aa}-m_{sc})\big]\Big|=O_\prec(N^{-8}),
\end{align}
where we used that $|G_{ij}| \prec 1$.  Moreover, there are at most $O(\log N)$ subleading terms consisting of four Green function entries generated in repeating (\ref{example_step_5}), each of which contributes $O_\prec(N\Psi^4)$ using the local law in (\ref{G}). In this way, we have improved the estimate as claimed in (\ref{small_goal}).

We remark that each term that contributes $O_\prec( N\Psi^4)$ generated in the above procedure shares a similar structure as the example term $I_1$ in (\ref{example_step_0}). Hence as in (\ref{example_step_2}), we can apply cumulant expansions  to these terms via the new summation index $k$ instead and improve the estimate iteratively. To see this, we introduce an abstract form of averaged products of shifted Green function entries \ie $G_{ij}-\delta_{ij} m_{sc}$ and entries of the matrix powers of $T$, which can be expanded as in (\ref{example_step_2}).

\subsection{Abstract form of products of Green function entries}\label{subsec:abstract_form}

	We will use the general letters $j_1, j_2,\ldots$ to denote the summation indices $\eg v,a,b$ in (\ref{example_step_0}) that run from 1 to $N$. We use the letters $x_i$, $y_i$ or $w_l$ to denote in general the row and column index of a Green function entry $G_{x_iy_i}$ or a shifted diagonal Green function entry $G_{w_lw_l}-m_{sc}$, and each $x_i,y_i,w_l$ represents some summation index in $j_1, j_2, \ldots$. We use $\equiv$ to denote such representation, \eg we write $x_i\equiv j_1$ if the row index $x_i$ represents the summation index $j_1$. 
	
	In order to avoid confusion, we clarify that $x_i \equiv j_1, ~y_i \equiv j_1$ means that both $x_i$ and $y_i$ represent the same summation index~$j_1$ and we write $x_i=y_i$ and $G_{x_iy_i}$ is a {\it diagonal} entry. If~$x_i$ and~$y_i$ represent two distinct summation indices, \eg $x_i \equiv j_1$ and $y_i \equiv j_2$, then we say $x_i \neq y_i$ and $G_{x_iy_i}$ is an {\it off-diagonal} entry. We remark that the value of $x_i$ could coincide with $y_i$ as the summation indices $j_1$ and $j_2$ run from $1$ to $N$. 
	
\begin{definition}\label{def_form}
	For any $d \in \N$ with $d\geq 2$, we use $\mathcal{J}_d:=(j_1,j_2,\ldots, j_{d})$ to denote a set of $d$ ordered summation indices ranging from $1$ to $N$. For any $2\leq m_0 \leq d$, the first $m_0$ summation indices $j_1,\ldots,j_{m_0}$ have the non-uniform weights from $\prod_{p=1}^{m_0-1}(T^{k_p})_{j_{p}j_{p+1}}~(k_p\geq 1)$, and each of the remaining $d-m_0$ indices $j_{m_0+1},\ldots,j_{d}$ has a uniform weight $N^{-1}$ in the summation.

	We split the summation index set $\mathcal{J}_d$ into two disjoint subsets $\mathcal{J}^{(\mathrm{off})}_d$ and $\mathcal{J}^{(\mathrm{diag})}_d$, with $\# \mathcal{J}^{(\mathrm{off})}_d=m_1$, $\#\mathcal{J}^{(\mathrm{diag})}_d=m_2$ and $d=m_1+m_2$. We use $\prod_{i=1}^{m_1} G_{x_i y_i}$ to denote a product of $m_1$ off-diagonal Green function entries, where each row and column index $x_i, y_i$ represents an element in $\mathcal{J}^{(\mathrm{off})}_d$ with $x_i \neq y_i$, and each element in $\mathcal{J}^{(\mathrm{off})}_d$ appears exactly twice in $\{x_i,y_i\}_{i=1}^{m_1}$. Moreover, we use $\prod_{l=1}^{m_2} \big(G_{w_lw_l}-m_{sc}\big)$ to denote a product of $m_2$ shifted diagonal Green function entries where each $w_l$ represents a different element in $\mathcal{J}^{(\mathrm{diag})}_d$.

	Then we consider the abstract form of $d$ (shifted) Green function entries
	\begin{align}\label{form}
			\frac{1}{N^{d-m_0}} \sum_{j_1, j_2, \ldots, j_{d}} \prod_{p=1}^{m_0-1} (T^{k_p})_{j_{p}j_{p+1}}  \Big( \prod_{i=1}^{m_1} G_{x_i y_i} \prod_{l=1}^{m_2} \big(G_{w_lw_l}-m_{sc}\big)(t,z) \Big),
	\end{align}
for any $z \in {\mathcal S}_{\mathrm{edge}}$ given in (\ref{S_edge}) and $t \geq 0$. The number $d$ is referred to the {\it degree} of such a term. A term of the form in (\ref{form}) with degree $d$ will be denoted by $P_{d} \equiv P_{d}(t,z)$ in general.  The collection of all the terms of the form in (\ref{form}) with degree $d$ is denoted by $\mathcal{P}_{d} \equiv \mathcal{P}_{d}(t,z)$.
\end{definition}

Below are some example terms of the form in (\ref{form}): 
\begin{align}\label{term}
&\sum_{a,b}  \E \big[(T^2)_{ab} G_{ab} G_{ba} \big] \in \mathcal{P}_{2}; \qquad 	\frac{1}{N} \sum_{a,b,v}  \E \Big[T_{ab} G_{vb} G_{bv} \big(G_{aa}-m_{sc}\big) \Big],~\frac{1}{N} \sum_{a,b,v}  \E \Big[T_{ab}G_{va} G_{ab} G_{bv}  \Big] \in \mathcal{P}_{3};\nonumber\\
&  	\frac{1}{N} \sum_{a,b,k,v}  \E \Big[T_{ab} (T^2)_{ka}G_{vb} G_{bv} \big(G_{aa}-m_{sc}) \big(G_{kk}-m_{sc}\big)  \Big],~\frac{1}{N^2} \sum_{a,b,k,v}    \E \big[(T^4)_{ab} G_{ka}G_{ka}G_{vb} G_{bv}\Big] \in \mathcal{P}_{4}.
\end{align}
where the terms in $\mathcal{P}_{3}$ are indeed the target terms to be estimated in (\ref{goal}) with the index $v$ having a uniform weight in the summation. 

Given a general term $P_d$ of the form in (\ref{form}), using the local law in (\ref{G}) and the max norm of the matrix $T^{k_p}$ in (\ref{prop_T}), we obtain 
$$|P_d| \prec \frac{1}{N^{d-m_0}}  \Big(\frac{C_0}{N}\Big)^{m_0-1} (N\Psi)^{m_2} \sum_{\mathcal{J}^{(\mathrm{off})}_d} \prod_{i=1}^{m_1} |G_{x_i y_i}|,$$
with $C_0$ given in (\ref{prop_T}). When all the $m_1$ indices in $\mathcal{J}^{(\mathrm{off})}_d$ have distinct values in the summation, the product of Green function entries can be bounded by $\Psi^{m_1}$ since $x_i$ and $y_i$ represent different summation indices. If two indices in $\mathcal{J}^{(\mathrm{off})}_d$ coincide in the summation, then the resulting product of Green function entries may be bounded by $\Psi^{m_1-2}$ only, however we gain an additional $N^{-1}$ from the index coincidence. Since $N^{-\frac{1}{3}+\epsilon} \leq \Psi \leq N^{-\epsilon}$, the terms with index coincidences are much smaller. Using that $d=m_1+m_2$, $2\leq m_0\leq d$, we obtain a naive estimate for any $P_{d} \in \mathcal{P}_{d}$, \ie
\begin{align}\label{initial}
	|P_{d}| \prec N (C_0\Psi)^d \ll N  \big( \frac{1}{N^{\epsilon/2}}\big)^{d}.
\end{align}
Moreover, if $\max_{1\leq p\leq m_0-1} \{k_p\} \geq K$ with $K$ given in (\ref{K_choice}), using \eqref{high_power_T} in Lemma \ref{lemma_prop_T}, we have
\begin{align}\label{large_K_term}
	|P_{d}| \prec \frac{1}{N^{d-m_0}} (1-c_0)^{K} \Big(\frac{C_0}{N}\Big)^{m_0-2} (N\Psi)^{m_2} \sum_{\mathcal{J}^{(\mathrm{off})}_d} \prod_{i=1}^{m_1} |G_{x_i y_i}| \prec (1-c_0)^{K} N^2 \leq N^{-8}.
\end{align}

Next we will apply cumulant expansions repeatedly in combination with (\ref{large_K_term}) to improve the naive estimate of $P_d$ in (\ref{initial}). We first note that both the summation index $j_1$ and $j_{m_0}$ in (\ref{form}) are special, since they appear only once in the non-uniform weight function $\prod_{p=1}^{m_0-1} (T^{k_p})_{j_{p}j_{p+1}}$.  If the product of Green function entries was independent of the special summation index $j_1$ or $j_{m_0}$, then the resulting term vanishes due to the summation property in (\ref{prop_T}).  Such decoupling can be realized using cumulant expansions via the index $j_1$ or $j_{m_0}$ as in (\ref{example_step_2}).

Now we introduce an expansion mechanism using either of the special index $j_1$ or $j_{m_0}$. Since both the matrix $T$ and $S$ are symmetric and they commute, we choose $j_1$ conventionally to perform cumulant expansions. We split the discussion into the following two cases, whose proofs are postponed till the next subsection. 
\begin{enumerate}
\item[(a)] If the special summation index $j_1$ appears in a shifted diagonal Green function entry, \ie $G_{j_1j_1}-m_{sc}$, then we have
\begin{align}\label{expand_case_diagonal}
	\E[P_d]=&\sum_{P_{d'} \in \mathcal{P}_{d'}, d'= d+1} \E[P_{d'}]+O_\prec(N^{-8}),
\end{align}
where the first group of terms contains at most $4Kd$ terms of the form in (\ref{form}), denoted by $P_{d'}$ in general, with higher degrees $d+1$ (\ie each term consists of $d+1$ shifted Green function entries) with possible factors $m_{sc}$ and $1-\ee^{-t}$. The last error $O_\prec(N^{-8})$ is from (\ref{large_K_term}) after iterative expansions. 

\item[(b)] If the special summation index $j_1$ appears in two different off-diagonal Green function entries, say $G_{x_1 y_1}$ and $G_{x_2 y_2}$ with $x_1,x_2 \equiv j_1$ and $y_1,y_2 \not\equiv j_1$ from (\ref{symmetric}), then we have
\begin{align}\label{expand_case_offdiagonal}
		\Im	\E[P_d]	=&\sum_{P_{d''} \in \mathcal{P}_{d''}, d''= d+1} \Im \E[P_{d''}]+O_{\prec}(N^{-8})\nonumber\\
		&+O_{\prec}\Big( K \big( ( C_0\Psi)^{d-3} \wedge 1\big)  \big(\E[\Im m_N]+\Im m_{sc} \big) \Big) \one_{m_1= 2}+O_\prec(K N^{-\frac{\epsilon}{2}} \E[\Im m_N]),
\end{align}
	with $m_1$ the number of off-diagonal Green function factors in the original $P_d$ in (\ref{form}). Here the first group of terms on the right side of (\ref{expand_case_offdiagonal}) contains at most $4Kd$ terms of the form in (\ref{form}), denoted by $P_{d''}$ in general, with higher degrees $d+1$. The first error term $O(N^{-8})$ is from (\ref{large_K_term}), and the errors on the last line of (\ref{expand_case_offdiagonal}) are from the cases with index coincidences; see the last term with $\delta_{j_1 y_1}$ in (\ref{expand_index_b_step1}) later. To estimate such terms with index coincidences in general, it is necessary to consider the imaginary part of $P_d$ as discussed in (\ref{im_estiamte_2}) below.
\end{enumerate}

We remark that the error $O_\prec(N^{-8})$ in (\ref{expand_case_diagonal}) and (\ref{expand_case_offdiagonal}) can be made arbitrary small $N^{-C}$ depending on where we stop our iterations, \ie the choice of $K$ given in (\ref{K_choice}).

Now we are ready to prove (\ref{goal}) by looking at 
\begin{align}\label{target}
	\frac{1}{N} \sum_{v,a,b}  T_{ab}\E \Big[\Im \big( G_{vb} G_{bv} (G_{aa}-m_{sc}) \big)\Big]\in \mathcal{P}_{3}, \qquad \frac{1}{N} \sum_{v,a,b}  T_{ab}\E \Big[\Im \big(G_{va} G_{ab} G_{bv} \big)\Big]\in \mathcal{P}_{3}.
\end{align}
For the first term in (\ref{target}), using the first expansion in (\ref{expand_case_diagonal}) via the index $a$, we obtain
\begin{align}\label{result_1}
	\frac{1}{N} \sum_{v,a,b}  T_{ab}\E \Big[\Im\big(  G_{vb} G_{bv} (G_{aa}-m_{sc}) \big)\Big]=\sum_{P_{d} \in \mathcal{P}_{4}} \Im \E[P_{d}]+O_{\prec}(N^{-8}),
\end{align}
where the first group of terms consists of at most $12K$ terms of the form in (\ref{form}) with degrees four. Similarly, we use the second expansion in (\ref{expand_case_offdiagonal}) to expand the second term in (\ref{target}) via the index $a$. Since the number of off-diagonal Green function entries, $m_1$, is three, we have
\begin{align}\label{result_2}
	\frac{1}{N} \sum_{v,a,b}  T_{ab}\E \Big[\Im \big(G_{va} G_{ab} G_{bv} \big)\Big]=\sum_{P_{d} \in \mathcal{P}_{4}} \Im \E[P_{d}]+O_{\prec}(N^{-8})+O_\prec(K N^{-\frac{\epsilon}{2}} \E[\Im m_N]),
\end{align}
where the last error is from the index coincidence \ie $v=a$ or $b=a$.

Next we improve the estimate of $\Im \E[P_d]$ for $d\geq 4$ using iterative expansions. For a general term $P_d$ of the form in (\ref{form}) with degree $d \geq 4$, combining with (\ref{expand_case_diagonal}) and (\ref{expand_case_offdiagonal}), we have
\begin{align}\label{expand_mechanism}
	\Im	\E[P_d]	=&\sum_{P_{d_1} \in \mathcal{P}_{d_1}, d_1= d+1} \Im \E[P_{d_1}] +O_\prec\Big( K N^{-\frac{\epsilon}{2}} \big(\E[\Im m_N]+\Im m_{sc} \big) \Big)+O_{\prec}(N^{-8}),
\end{align}
where the first group of terms is a linear combination of at most $4Kd$ terms with degrees $d+1$, denoted by $\Im \E[P_{d_1}]$ in general for the first iteration step.  We further expand each resulting $\Im \E[P_{d_1}]$ and obtain from (\ref{expand_mechanism}) that
\begin{align}\label{expand_mechanism_second}
	\Im	\E[P_d]	=&\sum_{P_{d_2} \in \mathcal{P}_{d_2}, d_2= d+2} \Im \E[P_{d_2}] +O_\prec(4K^2d N^{-\frac{\epsilon}{2}}) \big(\E[\Im m_N]+\Im m_{sc} \big)+O_{\prec}(4Kd N^{-8}),
\end{align}
where the first group of terms is a linear combination of at most $(4K)^2 d(d+1)$ terms of the form in (\ref{form}) with degrees $d+2$, which are denoted by $\Im \E[P_{d_2}]$ in general for the second step. Iterating the above process for $D-d$ times, we then obtain that
\begin{align}\label{expand_mechanism_third}
	\Im	\E[P_d]	=&\sum_{P_{d'} \in \mathcal{P}_{d'}, d'= D} \Im \E[P_{d'}] +O_\prec\Big((4KD)^{D} N^{-\frac{\epsilon}{2}}\Big) \big(\E[\Im m_N]+\Im m_{sc} \big)+O_{\prec}\big((4KD)^{D}N^{-8}\big),
\end{align}
where the first group of terms contains at most $(4KD)^D$ terms, denoted by $\Im \E[P_{d'}]$ in general, of the form in (\ref{form}) with degrees $D$. We now choose a sufficiently large but fixed $D$ depending only on $\epsilon$ such that $D \geq \frac{4}{\epsilon}$. Using the naive estimate of $P_d$ in (\ref{initial}), the estimate of $\Im m_{sc}$ in (\ref{22}) and that $K=O(\log N)$, for any $d\geq 4$ we have,
\begin{align}\label{four}
\Im	\E[P_d]=&O_{\prec}\Big( (4KD)^{D} \big(N (C_0\Psi)^D+N^{-8}\big)\Big)+O_\prec\Big((4KD)^{D} N^{-\frac{\epsilon}{2}}\Big) \big(\E[\Im m_N]+\Im m_{sc} \big)\nonumber\\
	=&O_{\prec}\big( N^{-\frac{\epsilon}{4}} \E[\Im m_N] \big)+O_\prec(N^{-\frac{1}{3}-\frac{\epsilon}{4}}),
\end{align}
uniformly for any $t\in \R^+$ and $z \in {\mathcal S}_{\mathrm{edge}}$.

Combining the improved estimates in (\ref{four}) for $d\geq 4$ with (\ref{result_1}) and (\ref{result_2}), we hence prove the estimate in (\ref{goal}) and finish the proof of Proposition \ref{lemma_step_1}. 
\end{proof}

\subsection{Proof of the expansions in (\ref{expand_case_diagonal}) and (\ref{expand_case_offdiagonal})}\label{subsec:novel_expansion}

We first prove the expansion in (\ref{expand_case_diagonal}) in Case 1 and next show the expansion in (\ref{expand_case_offdiagonal}) in Case 2.

\textbf{Case 1: index $j_1$ appearing in $G_{j_1j_1}-m_{sc}$.} Given a term in (\ref{form}), without loss of generality, we may assume that $w_1 \equiv j_1$.  Using the relation $-\frac{1}{m_{sc}}=z+m_{sc}$ and the resolvent identity $zG_{j_1j_1}=(HG)_{j_1j_1}-1$, we have \cf (\ref{example_step_1}),
\begin{align}\label{expand_index_a_step1}
	-\frac{1}{m_{sc}}\E[P_d]
	=&\E \Big[	\frac{1}{N^{d-m_0}} \sum_{j_1, \ldots, j_d} \prod_{p=1}^{m_0-1} (T^{k_p})_{j_{p}j_{p+1}} \sum_{k} h_{j_1k}G_{kj_1}  \prod_{i=1}^{m_1} G_{x_i y_i} \prod_{l=2}^{m_2} \big(G_{w_lw_l}-m_{sc}\big)\Big] \nonumber\\
	&+m_{sc}\E \Big[	\frac{1}{N^{d-m_0}} \sum_{j_1, \ldots, j_d} \prod_{p=1}^{m_0-1} (T^{k_p})_{j_{p}j_{p+1}} G_{j_1j_1}  \prod_{i=1}^{m_1} G_{x_i y_i} \prod_{l=2}^{m_2} \big(G_{w_lw_l}-m_{sc}\big)\Big].
\end{align}
Using the differentiation rule in (\ref{dH}) and the definition of $S_{j_1k}(t)$ given in (\ref{S_ab}), we apply cumulant expansions to the first line of (\ref{expand_index_a_step1}) and obtain
\begin{align}\label{expand_index_a_step2}
	-\frac{1}{m_{sc}}\E[P_d]	=&-\frac{1}{N^{d-m_0}} \sum_{j_1, \ldots, j_d,k} \prod_{p=1}^{m_0-1} (T^{k_p})_{j_{p}j_{p+1}} S_{j_1 k}(t) \E \Big[  G_{j_1j_1}G_{kk}  \prod_{i=1}^{m_1} G_{x_i y_i} \prod_{l=2}^{m_2} \big(G_{w_lw_l}-m_{sc}\big)\Big]\nonumber\\
	&-\frac{1}{N^{d-m_0}} \sum_{j_1, \ldots, j_d,k} \prod_{p=1}^{m_0-1} (T^{k_p})_{j_{p}j_{p+1}}S_{j_1 k}(t) \E \Big[  G_{kj_1}G_{kj_1}  \prod_{i=1}^{m_1} G_{x_i y_i} \prod_{l=2}^{m_2} \big(G_{w_lw_l}-m_{sc}\big)\Big]\nonumber\\
	&+\frac{1}{N^{d-m_0}} \sum_{j_1, \ldots, j_d,k} \prod_{p=1}^{m_0-1} (T^{k_p})_{j_{p}j_{p+1}} S_{j_1 k}(t)\E \Big[ G_{kj_1} \frac{\partial   \prod_{i=1}^{m_1} G_{x_i y_i} \prod_{l=2}^{m_2} \big(G_{w_lw_l}-m_{sc}\big)}{\partial h_{j_1 k}}\Big]\nonumber\\
	&+m_{sc}\E \Big[	\frac{1}{N^{d-m_0}} \sum_{j_1, \ldots, j_d} \prod_{p=1}^{m_0-1} (T^{k_p})_{j_{p}j_{p+1}} G_{j_1j_1}  \prod_{i=1}^{m_1} G_{x_i y_i} \prod_{l=2}^{m_2} \big(G_{w_lw_l}-m_{sc}\big)\Big]\nonumber\\
	&=:A_{1}+A_{2}+A_{3}+A_4,
\end{align}
where the terms on the first two lines, \ie $A_1$ and $A_2$ are from acting $\partial/\partial h_{j_1k}$ on $G_{kj_1}$ using (\ref{dH}). 

We first consider the term $A_2$ on the second line above. Using the first relation in (\ref{TS_prop}), the term $A_2$ can be split into two subterms, \ie
\begin{align}\label{split}
	A_2=&-(1-\ee^{-t})\frac{1}{N^{d-m_0}} \sum_{j_1, \ldots, j_d,k} \prod_{p=1}^{m_0-1} (T^{k_p})_{j_{p}j_{p+1}}T_{k j_1 } \E \Big[  G_{kj_1}G_{kj_1}  \prod_{i=1}^{m_1} G_{x_i y_i} \prod_{l=2}^{m_2} \big(G_{w_lw_l}-m_{sc}\big)\Big]\nonumber\\
	&-\frac{1}{N^{d-m_0+1}} \sum_{j_1, \ldots, j_d,k} \prod_{p=1}^{m_0-1} (T^{k_p})_{j_{p}j_{p+1}} \E \Big[  G_{kj_1}G_{kj_1}  \prod_{i=1}^{m_1} G_{x_i y_i} \prod_{l=2}^{m_2} \big(G_{w_lw_l}-m_{sc}\big)\Big],
\end{align}  
where the second subterm is clearly of the form in (\ref{form}) with degree $d+1$, and after properly renaming and rearranging the summation indices, the first subterm is also of the form in (\ref{form}) with degree $d+1$ and a factor $1-\ee^{-t}$. 

We next look at the third line of (\ref{expand_index_a_step2}), denoted by $A_3$. Similarly as in (\ref{split}), each resulting term in $A_3$ using the differentiation rule in (\ref{dH}) can be split into two subterms of the form in (\ref{form}) with degree $d+1$, since both the fresh index $k$ and the index $j_1$ do not appear in $\{x_i,y_i\}_{i=1}^{m_1} \cup \{w_l\}_{l=2}^{m_2}$. Therefore, $A_2+A_3$ is indeed a linear combination of at most $4d-2$ terms of the form in (\ref{form}) with higher degrees $d+1$.

Next, we observe a cancellation between the terms on the first and fourth line of (\ref{expand_index_a_step2}), denoted by $A_1$ and $A_4$ respectively. Since $S(t)$ is doubly stochastic,  $A_{1}+A_4$ is given by
\begin{align}\label{leading_term_expand}
	-\frac{1}{N^{d-m_0}} \sum_{j_1, \ldots, j_d,k} \prod_{p=1}^{m_0-1} (T^{k_p})_{j_{p}j_{p+1}} S_{k j_1}(t)\E \Big[  G_{j_1j_1} (G_{kk}-m_{sc})  \prod_{i=1}^{m_1} G_{x_i y_i} \prod_{l=2}^{m_2} \big(G_{w_lw_l}-m_{sc}\big)\Big].
\end{align}
We further split this term into two terms using $G_{j_1j_1}=m_{sc}+\big( G_{j_1j_1}-m_{sc}\big)$. The resulting term corresponding to $G_{j_1j_1}-m_{sc}$ has a higher degree $d+1$ and can be split into two subterms of the form in (\ref{form}) as in (\ref{split}). Together with all the terms from $A_2+A_3$, we obtain at most $4d$ terms of the form in (\ref{form}) with higher degrees $d+1$, in general denoted by
\begin{align}\label{higher_term}
	\sum_{P_{d_1} \in \mathcal{P}_{d_1}, d_1 = d+1} \E[P_{d_1}],
\end{align}
where the subscript $1$ in $d_1$ indicates the first step of expansion, and each term $\E[P_{d_1}]$ comes with possible factors $m_{sc}$ and $1-\ee^{-t}$. We often ignored these coefficients for notational simplicity.

We next estimate the leading term corresponding to replacing $G_{j_1j_1}$ in (\ref{leading_term_expand}) with $m_{sc}$, \ie
\begin{align}\label{cf_leading_2}
	&-\frac{m_{sc}}{N^{d-m_0}} \sum_{j_1, \ldots, j_d,k} \prod_{p=1}^{m_0-1} (T^{k_p})_{j_{p}j_{p+1}} S_{k j_1}(t)\E \Big[  (G_{kk}-m_{sc})  \prod_{i=1}^{m_1} G_{x_i y_i} \prod_{l=2}^{m_2} \big(G_{w_lw_l}-m_{sc}\big)\Big]\nonumber\\
	=&-\frac{m_{sc}}{N^{d-m_0}} \sum_{k,j_2, \ldots, j_d} (S(t) T^{k_1})_{kj_{2}} \prod_{p=2}^{m_0-1} (T^{k_p})_{j_{p}j_{p+1}}\E \Big[  (G_{kk}-m_{sc})  \prod_{i=1}^{m_1} G_{x_i y_i} \prod_{l=2}^{m_2} \big(G_{w_lw_l}-m_{sc}\big)\Big]\nonumber\\
=&-\frac{m_{sc}(1-\ee^{-t})}{N^{d-m_0}} \sum_{j_1,j_2, \ldots, j_d}  (T^{k_1+1})_{j_1 j_2}\prod_{p=2}^{m_0-1} (T^{k_p})_{j_{p}j_{p+1}} \E \Big[   (G_{kk}-m_{sc})  \prod_{i=1}^{m_1} G_{x_i y_i} \prod_{l=2}^{m_2} \big(G_{w_lw_l}-m_{sc}\big)\Big]
\end{align} 
where in the last step we used the second relation in (\ref{TS_prop}) and replaced the index $k$ with the original $j_1$. Compared to the original term $P_d$, the matrix power $k_1$ has been increased to $k_1+1$. We denote this term by $P_d(k_1 \rightarrow k_1+1)$.

Combining (\ref{higher_term}) and (\ref{cf_leading_2}) with (\ref{expand_index_a_step2}) and multiplying $-m_{sc}$ on both sides of (\ref{expand_index_a_step2}), we obtain
\begin{align}\label{cf_expand}
	\E[P_d]=&m^2_{sc}(1-\ee^{-t}) \E[P_d(k_1 \rightarrow k_1+1)]+\sum_{P_{d_1} \in \mathcal{P}_{d_1}, d_1 = d+1} \E[P_{d_1}],
\end{align}
where the second group of terms contains at most $4d$ terms of the form in (\ref{form}) with higher degrees $d+1$ with possible factors $m_{sc}$ and $1-\ee^{-t}$. These deterministic factors are not harmful, since $|1-\ee^{-t}|\leq 1,~|m_{sc}(z)|\leq 1$ uniformly in $t\in \R^+$ and $z\in {\mathcal S}_{\mathrm{edge}}$.

We continue to expand the leading term $P_d(k_1\rightarrow k_1+1)$ in the same way and obtain from (\ref{cf_expand}) that
\begin{align}
	\E[P_d]=&m^4_{sc}(1-\ee^{-t})^2\E[P_d(k_1 \rightarrow k_1+2)]+\sum_{P_{d_2} \in \mathcal{P}_{d_1}, d_2= d+1}  \E[P_{d_2}],
\end{align}
where the second group of terms contains at most $8d$ terms of the form in (\ref{form}) with higher degrees $d+1$, together with the terms $P_{d_1}$ from the first step in (\ref{cf_expand}). These terms are denoted by $P_{d_2}$ in general for the second step of expansion. We iterate the above process for $K-k_1$ times until the matrix power $k_1$ is increased to sufficiently large $K$ chosen in (\ref{K_choice}). Note that from (\ref{large_K_term}), we have
$$\E[ P_d(k_1 \rightarrow K)]=O_\prec(N^{-8}).$$
Hence we arrive at
\begin{align}\label{expand_case_1}
	\E[P_d]=&\sum_{P_{d_{K-k_1}} \in \mathcal{P}_{d_{K-k_1}}, d_{K-k_1}= d+1} \E[P_{d_{K-k_1}}]+O_\prec(N^{-8}),
\end{align}
where the first group of terms is a linear combination of at most $4(K-k_1)d$ terms generated in the iterations that are of the form in (\ref{form}) with higher degrees $d+1$. This proves the expansion in (\ref{expand_case_diagonal}).

\textbf{Case 2: index $j_1$ appearing in two off-diagonal Green function entries.}  Given a term in (\ref{form}), without loss of generality, we may assume that $x_1,x_2 \equiv j_1$ and $y_1,y_2 \not\equiv j_1$ from (\ref{symmetric}).  Using the relation $-\frac{1}{m_{sc}}=z+m_{sc}$ and the resolvent identity $zG_{j_1y_1}=(HG)_{j_1y_1}-\delta_{j_1 y_1}$, we have
\begin{align}\label{expand_index_b_step1}
	-\frac{1}{m_{sc}}\E[P_d]=&\E \Big[	\frac{1}{N^{d-m_0}} \sum_{j_1, \ldots, j_d} \prod_{p=1}^{m_0-1} (T^{k_p})_{j_{p}j_{p+1}} \sum_{k} h_{j_1k} G_{ky_1} G_{j_1y_2} \prod_{i=3}^{m_1} G_{x_i y_i} \prod_{l=1}^{m_2} \big(G_{w_lw_l}-m_{sc}\big)\Big]\nonumber\\
	&+m_{sc}\E \Big[	\frac{1}{N^{d-m_0}} \sum_{j_1, \ldots, j_d} \prod_{p=1}^{m_0-1} (T^{k_p})_{j_{p}j_{p+1}} G_{j_1y_1} G_{j_1y_2} \prod_{i=3}^{m_1} G_{x_i y_i} \prod_{l=1}^{m_2} \big(G_{w_lw_l}-m_{sc}\big)\Big] \nonumber\\
	& +\E \Big[	\frac{1}{N^{d-m_0}} \sum_{j_1, \ldots, j_d} \prod_{p=1}^{m_0-1} (T^{k_p})_{j_{p}j_{p+1}} \delta_{j_1y_1} G_{j_1y_2} \prod_{i=3}^{m_1} G_{x_i y_i} \prod_{l=1}^{m_2} \big(G_{w_lw_l}-m_{sc}\big)\Big]\nonumber\\
	=:&B_1+B_2+B_3.
\end{align}

Applying cumulant expansions to $B_1$, using the differentiation rule in (\ref{dH}) and the definition of $S(t)$ in (\ref{S_ab}), we have as in (\ref{expand_index_a_step2}) that
\begin{align}\label{temp_111}
B_1
=&	-\frac{1}{N^{d-m_0}} \sum_{j_1, \ldots, j_d,k} \prod_{p=1}^{m_0-1} (T^{k_p})_{j_{p}j_{p+1}} S_{j_1 k}(t) \E \Big[  G_{j_1y_1}G_{kk} G_{j_1y_2} \prod_{i=3}^{m_1} G_{x_i y_i}  \prod_{l=1}^{m_2} \big(G_{w_lw_l}-m_{sc}\big)\Big]\nonumber\\
&-\frac{1}{N^{d-m_0}} \sum_{j_1, \ldots, j_d,k} \prod_{p=1}^{m_0-1} (T^{k_p})_{j_{p}j_{p+1}}S_{j_1 k}(t)  \E \Big[  G_{ky_1}G_{j_1j_1}  G_{ky_2} \prod_{i=3}^{m_1} G_{x_i y_i} \prod_{l=1}^{m_2} \big(G_{w_lw_l}-m_{sc}\big)\Big]\nonumber\\
&+\mbox{remaining terms}\nonumber\\
=:&B^{(1)}_{1}+B^{(2)}_{1}+\mbox{remaining terms},
\end{align}
where the first two leading terms $B^{(1)}_{1}$ and $B^{(2)}_{1}$ are from acting $\frac{\partial}{\partial h_{j_1k}}$ on the off-diagonal Green function entries $G_{k y_1}$ and $G_{j_1 y_2}$ respectively, and all the remaining terms have higher degrees $d+1$ since the index $j_1$ and $k$ do not appear in $\{x_i,y_i\}_{i=3}^{m_1} \cup \{w_l\}_{l=1}^{m_2}$. Moreover, each of these remaining terms can be split into two subterms of the form in (\ref{form}) as in (\ref{split}).

Since $S(t)$ is a doubly stochastic matrix, there is a cancellation between the first leading term $B^{(1)}_{1}$ in (\ref{temp_111}) and the second term $B_2$ in (\ref{expand_index_b_step1}), \ie
\begin{align}\label{B_11+B_2}
	B^{(1)}_{1}+B_2=&-\frac{1}{N^{d-m_0}} \sum_{j_1, \ldots, j_d,k} \prod_{p=1}^{m_0-1} (T^{k_p})_{j_{p}j_{p+1}} S_{j_1 k}(t) \E \Big[  G_{j_1y_1} (G_{kk}-m_{sc}) G_{j_1y_2} \prod_{i=3}^{m_1} G_{x_i y_i}  \prod_{l=1}^{m_2} \big(G_{w_lw_l}-m_{sc}\big)\Big],
\end{align}
where the resulting term can be further split into two subterms of the form in (\ref{form}) with higher degree $d+1$ as in (\ref{split}).

We next look at the other leading term $B^{(2)}_{1}$ in (\ref{temp_111}). As explained below (\ref{leading_term_expand}) and in (\ref{cf_leading_2}), using the second relation in (\ref{TS_prop}), this term can be written as 
\begin{align}\label{B_3}
	B^{(1)}_{2}=&-\frac{m_{sc}(1-\ee^{-t})}{N^{d-m_0}} \sum_{j_2, \ldots, j_d,k} (T^{k_1+1})_{kj_{2}} \prod_{p=2}^{m_0-1} (T^{k_p})_{j_{p}j_{p+1}} \E \Big[  G_{ky_1} G_{ky_2} \prod_{i=3}^{m_1} G_{x_i y_i} \prod_{l=1}^{m_2} \big(G_{w_lw_l}-m_{sc}\big)\Big]\nonumber\\
	&-\frac{1}{N^{d-m_0+1}} \sum_{j_1, \ldots, j_d,k} \prod_{p=1}^{m_0-1} (T^{k_p})_{j_{p}j_{p+1}} \E \Big[  G_{ky_1} (G_{j_1j_1}-m_{sc})  G_{ky_2} \prod_{i=3}^{m_1} G_{x_i y_i} \prod_{l=1}^{m_2} \big(G_{w_lw_l}-m_{sc}\big)\Big].
\end{align}
 Note that after replacing the fresh index $k$ with the original $j_1$, the leading term on the right side of (\ref{B_3}) is actually the original term $P_d$ with the matrix power $k_1$ increased to $k_1+1$, denoted by $P_d(k_1 \rightarrow k_1+1)$ up to a factor $-m_{sc}(1-\ee^{-t})$. Moreover, the remaining term on the last line of (\ref{B_3}) can be split into two subterms of the form in (\ref{form}) with higher degree $d+1$ as in (\ref{split}). The collection of all the terms of the form in (\ref{form}) with higher degrees $d+1$ from the remaining terms in (\ref{temp_111}), the term in (\ref{B_11+B_2}) and the last line of (\ref{B_3}), is then in general denoted by
\begin{align}\label{higher_term2}
	\sum_{P_{d_1} \in \mathcal{P}_{d_1}, d_1 = d+1} \E[P_{d_1}],
\end{align}
where the sum contains at most $4d$ terms of the form in (\ref{form}), and we ignored the possible uniformly bounded factors $m_{sc}$ and $1-\ee^{-t}$ for notational simplicity. 

To sum up, multiplying $-m_{sc}$ on both sides of (\ref{expand_index_b_step1}), we obtain
\begin{align}\label{some_step}
	\E[P_d]=m^2_{sc}(1-\ee^{-t})\E[P_d(k_1 \rightarrow k_1+1)]+\sum_{P_{d_1} \in \mathcal{P}_{d_1}, d_1 = d+1} \E[P_{d_1}]-m_{sc} B_3,
\end{align}
where $B_3$ is the last term with the index coincidence $\delta_{j_1y_1 }$ on the right side of (\ref{expand_index_b_step1}). The rest of this subsection is devoted to estimating $B_3$, \ie
$$B_3=\E \Big[	\frac{1}{N^{d-m_0}} \sum_{j_1, \ldots, j_d} \prod_{p=1}^{m_0-1} (T^{k_p})_{j_{p}j_{p+1}} \delta_{j_1y_1} G_{j_1y_2} \prod_{i=3}^{m_1} G_{x_i y_i} \prod_{l=1}^{m_2} \big(G_{w_lw_l}-m_{sc}\big)\Big].$$

We assume $y_1$ represents the summation index $j_{q}~(2 \leq q \leq d)$, \ie $y_1 \equiv j_q$ and then $\delta_{j_1y_1}=\delta_{j_1j_q}$. Using the local law in (\ref{G}) and the max norm of $T^{k_p}$ in (\ref{prop_T}), we have
\begin{align}\label{diagona_estimate}
	|B_3| \prec \frac{1}{N^{d-m_0}} \Big( \frac{C_0}{N}\Big)^{m_0-1} \Psi^{m_2} \sum_{j_2,\ldots, j_d} \big|G_{j_q y_2}\big|\prod_{i=3}^{m_1} \big|G_{x_i y_i}\big|.
\end{align}
We split the discussion into three cases:
\begin{enumerate}
	\item If $m_1=2$ which implies that $y_1,y_2 \equiv j_q$, then we have from (\ref{diagona_estimate})
		\begin{align}\label{im_estiamte_1}
			|B_3| \prec 	\frac{1}{N^{d-m_0}} \Big( \frac{C_0}{N}\Big)^{m_0-1} \Psi^{m_2} \sum_{j_2,\ldots, j_d} \big|G_{j_q j_q}\big| \prec \big(C_0\Psi\big)^{d-2},
		\end{align}
	which is in general not enough to reach (\ref{expand_case_offdiagonal}). However after taking the imaginary part, we have
	\begin{align}\label{im_estiamte_2}
		\Im B_3 = &\E \Big[	\frac{1}{N^{d-m_0}} \sum_{j_2, \ldots, j_d} \prod_{p=1}^{m_0-1} (T^{k_p})_{j_{p}j_{p+1}}  \Im \Big(G_{j_q j_q}  \prod_{l=1}^{m_2} (G_{w_lw_l}-m_{sc})\Big)\Big]\nonumber\\
		\prec & \begin{cases}
			\E[\Im m_N], & d=2, \\
			\big(C_0\Psi\big)^{d-3} \big(\E[\Im m_N]+\Im m_{sc} \big), & d\geq 3\,,
		\end{cases} \quad \qquad d=2+m_2,
	\end{align}
	where we used the local law in (\ref{G}) and the estimate of $\Im G_{aa}$ in (\ref{ward_1}). We remark that the imaginary part is necessary for a general $P_d$, \eg the first example term of $P_d$ in (\ref{term})
	
	\item If $m_1 = 3$, then there exists a different summation index than $j_q$, say $j_{q'}$ such that $y_2 \equiv j_{q'}$ and $\prod_{i=1}^{3}G_{x_iy_i}=G_{j_1j_q}G_{j_1 j_{q'}} G_{j_{q} j_{q'}}$ from (\ref{symmetric}). Using the generalized Ward identity in (\ref{ward_2}) and that $d=3+m_2$, we have
		\begin{align}\label{im_estiamte_3}
				|B_3| \prec 	\frac{1}{N^{d-m_0}} \Big( \frac{C_0}{N}\Big)^{m_0-1} \Psi^{m_2} \sum_{j_2,\ldots, j_d} \E \big|G_{j_q j_{q'}}\big|^2 \prec \big(C_0\Psi\big)^{d-2} \E[\Im m_N] \leq N^{-\frac{\epsilon}{2}}\E[\Im m_N].
			\end{align}
	\item If $m_1 \geq 4$, then there exists at least $m_1-2$ off-diagonal Green function entries in (\ref{diagona_estimate}). Using the generalized Ward identity in (\ref{ward_2}) and that $d=4+m_2$, we have
	\begin{align}\label{im_estiamte_4}
		|B_3| \prec \big(C_0\Psi\big)^{d-3} \E[\Im m_N] \leq N^{-\frac{\epsilon}{2}}\E[\Im m_N].
	\end{align}
\end{enumerate}

We now return to (\ref{some_step}). Taking the imaginary part and using the estimates of $B_3$ in (\ref{im_estiamte_1})-(\ref{im_estiamte_4}), we obtain 
\begin{align}\label{expand_case_2}
	\Im \E[P_d]=&(1-\ee^{-t}) \Im\big( m^2_{sc} \E[P_d(k_1 \rightarrow k_1+1)]\big)+\sum_{{P_{d_1} \in \mathcal{P}_{d_1},d_1 = d+1}} \Im \E[P_{d_1}]\nonumber\\
	&+ O_{\prec}\Big( \big( ( C_0\Psi)^{d-3} \wedge 1\big)  \big(\E[\Im m_N]+\Im m_{sc} \big) \Big)\one_{m_1=2}+O_\prec(N^{-\frac{\epsilon}{2}} \E[\Im m_N]),
\end{align}
where the second group of terms contains at most $4d$ terms of the form in (\ref{form}).

Iterating the above process for $K-k_1$ times until the power $k_1$ is raised to sufficiently large $K$ chosen in (\ref{K_choice}) and using the estimate in (\ref{large_K_term}), we have
\begin{align}\label{expand_case_22}
\Im	\E[P_d]	=&\sum_{P_{d_{K-k_1}} \in \mathcal{P}_{d_{K-k_1}}, d_{K-k_1}= d+1} \Im \E[P_{d_{K-k_1}}]+O_{\prec}(N^{-8})\nonumber\\
&+O_{\prec}\Big( K \big( ( C_0\Psi)^{d-3} \wedge 1\big)  \big(\E[\Im m_N]+\Im m_{sc} \big) \Big)\one_{m_1= 2}+O_\prec(K N^{-\frac{\epsilon}{2}} \E[\Im m_N]),
\end{align}
where the first group of terms contains at most $4(K-k_1)d$ terms of the form in (\ref{form}) with higher degrees $d+1$, and the errors on the last line are from the term $B_3$ in (\ref{temp_111}) with the index coincidences.  This proves the expansion in (\ref{expand_case_offdiagonal}).

\section{From Gaussian to Wigner with the same variance profile: Proof of Proposition \ref{GCT_mn_2}}\label{sec:step2}

In this section, we consider the matrix flow $H^{(2)}(t)$ given in (\ref{sum_1}) that interpolates between the generalized Gaussian matrix $W^{S}$ studied in Section \ref{sec:step1} and any generalized Wigner matrix with the same variance profile matrix $S$. To prove Proposition \ref{GCT_mn_2}, using Gr\"onwall's inequality in combination with the estimate (\ref{middle}) for the initial matrix $W^{S}$, it suffices to show
\begin{proposition}\label{lemma_step_2}
	For any $t\geq 0$ and $z \in {\mathcal S}_{\mathrm{edge}}$, we have
	$$		\Big|\frac{\dd \E[\Im m^{(2)}_N(t,z)]}{\dd t}\Big|\prec \ee^{-t} \big(N^{-\epsilon}\E[\Im m^{(2)}_N(t,z)]+N^{-1/2}\big).$$
\end{proposition}

We remark that the proof of Proposition \ref{lemma_step_2} is in the same spirit of \cite{Schnelli+Xu} to compare standard Wigner matrices with the corresponding Gaussian ensemble with the second moment matching.

\begin{proof}[Proof of Proposition \ref{lemma_step_2}]

Recall the matrix flow $H^{(2)}(t)= (h_{ab})_{1\leq a,b\leq N}$ in (\ref{sum_1}), \ie
\begin{align}
	H^{(2)}(t)=\mathrm{e}^{-\frac{t}{2}}W^{S} +\sqrt{1-\mathrm{e}^{-t}} H^{S},
\end{align}
where $W^{S}=(w^{(s)}_{ab})_{1\leq a,b\leq N}$ is a Gaussian matrix with the variance profile matrix $S$, and $H^{S}=(h^{(s)}_{ab})_{1\leq a,b\leq N}$ is any generalized matrix which is independent from $W^{S}$ with the same variance profile.

We remark that the local law in (\ref{G}), the estimate of $\Im G_{aa}$ in (\ref{ward_1}) and the generalized Ward identity in (\ref{ward_2}) also hold true for the resolvent of $H^{(2)}(t)$, denoted by $G^{(2)}(t,z)$. For notational simplicity we often ignore in this section the superscript $(2)$ and the dependence on $t \in \R^+, ~z \in \C^+$ in $G^{(2)}(t,z)$.

Taking the time derivative of  $\E[\Im m^{(2)}_N(t,z)]$, using the differentiation rules in (\ref{dH}) and cumulant expansions as in (\ref{intial_step}) and (\ref{sde_im_2}), we obtain
\begin{align}\label{sde_im}
	\frac{\dd}{\dd t}\E[\Im m^{(2)}_N(t,z)]
	=&-\frac{1}{N}\sum_{v=1}^N\sum_{a,b=1}^N \E\Big[\Im \Big( (G_{va} G_{bv}) \big(-\frac{\ee^{-\frac{t}{2}}}{2} w^{(s)}_{ab}+\frac{\ee^{-t}}{2\sqrt{1-\ee^{-t}}} h^{(s)}_{ab} \big)\Big) \Big]\nonumber\\
=&-\frac{\ee^{-t}}{2}\frac{1}{N} \sum_{k+1=3}^{4} \frac{1}{k!} \frac{s^{(k+1)}_{ab}(t)}{N^{\frac{k+1}{2}}}  \sum_{v,a,b} \E \Big[\Im \frac{\partial^k G_{va}G_{bv}}{\partial h_{ab}^k}\Big]+O_{\prec}(\frac{1}{\sqrt{N}}),
\end{align}
where $s^{(k+1)}_{ab}(t)$ is of the constant order and given by
\begin{align}\label{s_ab_t}
   s^{(k+1)}_{ab}(t)= (1-\ee^{-t})^{\frac{k-1}{2}} c^{(k+1)}(\sqrt{N}h^{(s)}_{ab}), \qquad k+1\geq 3,
\end{align}
with $c^{(k+1)}(\sqrt{N}h^{(s)}_{ab})$ defined in (\ref{cumulant_k}). By direct computations, the second order terms for $k+1=2$ vanish in (\ref{sde_im}) since the variance of $w^{(s)}_{ab}$ coincides with the variance of $h^{(s)}_{ab}$. The last error $O_\prec(N^{-1/2})$ stems from truncating the expansions at the fourth order, using the local law in (\ref{G}) and the moment condition in (\ref{moment_condition}). 

It then suffices to estimate the third and fourth order term in (\ref{sde_im}), \ie
\begin{align}\label{third+fourth_term}
	K_3:= \frac{1}{N^{\frac{5}{2}}} \sum_{v,a,b} s_{ab}^{(3)}(t)\E \Big[\Im \frac{\partial^2 G_{va}G_{bv}}{\partial h_{ab}^2}\Big]; \qquad 	K_4:=\frac{1}{N^{3}} \sum_{v,a,b} s_{ab}^{(4)}(t)\E \Big[\Im \frac{\partial^3 G_{va}G_{bv}}{\partial h_{ab}^3}\Big].
\end{align}

We first look at the fourth order term $K_4$. Using the differentiation rule in (\ref{dH}) and (\ref{symmetric}), 
 $K_4$ can be written as a linear combination of the following terms:
\begin{align}\label{some_term}
	\frac{1}{N^{3}} \sum_{v,a,b} s_{ab}^{(4)}(t)\E \big[\Im (G_{va}  G_{aa}G_{bb}G_{bb}G_{av})\big]; \qquad 	\frac{1}{N^{3}} \sum_{v,a,b} s_{ab}^{(4)}(t)\E \big[\Im (G_{va}  G_{aa} G_{bb} G_{ab} G_{bv} )\big];\nonumber\\
	\frac{1}{N^{3}} \sum_{v,a,b} s_{ab}^{(4)}(t)\E \big[\Im (G_{va}  G_{ab}G_{ab}G_{bb} G_{av})\big]; \qquad 	\frac{1}{N^{3}} \sum_{v,a,b} s_{ab}^{(4)}(t)\E \big[\Im (G_{va} G_{ab} G_{ab} G_{ab} G_{bv})\big].
\end{align}
 Using the local law in (\ref{G}) and the generalized Ward identity in~(\ref{ward_1})-(\ref{ward_2}), we have 
\begin{align}\label{fourth_estimate}
	|K_4|=O_{\prec}\Big(\frac{\E[\Im m^{(2)}_N(t,z)]}{N\eta}\Big).
\end{align}

We next estimate the third order term $K_3$ in (\ref{third+fourth_term}), which can be written into a linear combination of the following terms:
\begin{align}\label{fourth_term}
		\frac{\sqrt{N}}{N^{3}} \sum_{v,a,b} s_{ab}^{(3)}(t)\E \big[\Im &(G_{va}  G_{aa}G_{bb}G_{bv})\big]; \qquad 	\frac{\sqrt{N}}{N^{3}} \sum_{v,a,b} s_{ab}^{(3)}(t)\E \big[\Im (G_{va} G_{bb} G_{ab} G_{av} )\big];\nonumber\\
		& 	\frac{\sqrt{N}}{N^{3}} \sum_{v,a,b} s_{ab}^{(3)}(t)\E \big[\Im (G_{va}  G_{ab}G_{ab}G_{bv})\big].
\end{align}
These third order terms cannot be estimated as in (\ref{fourth_estimate}) since we are off by $\sqrt{N}$ from the third order cumulants. As observed in \cite{Schnelli+Xu}, these third order terms are the so-called unmatched terms (up to a factor $\sqrt{N}$) with unmatched indices $a$ and $b$; see Definition \ref{unmatch_def} below.

We can extend the expansion mechanism introduced in \cite[Section 6]{Schnelli+Xu} to generalized Wigner matrices with inhomogeneous variances. Due to the existence of unmatched indices, each time we perform the expansions via an unmatched index, we gain one more off-diagonal Green function entry in the leading terms, which will contribute additional $\Psi\leq N^{-\epsilon}$ from the local law. By invoking the expansions iteratively for $D$ times with $D>0$ being sufficiently large depending only on $\epsilon$, we prove in Proposition~\ref{unmatch_lemma} below that these unmatched terms in (\ref{fourth_term}) (without the factor $\sqrt{N}$) can be bounded by $O_\prec(N^{-1})$, where the error $N^{-1}$ is from the index coincidences, \eg $a=b$. 

Before we give the formal statement, we define the following abstract form of averaged products of Green function entries and possible shifted diagonal Green function factors $G_{vv}-m_{sc}$ that will be generated in the expansions as explained later. We remark that the abstract form below is slightly different than in \cite[Definition 4.2]{Schnelli+Xu} to adapt to generalized Wigner matrices with inhomogeneous variances.

\begin{definition}
	For any fixed $m \in \N$, we use  $\mathcal{I}_m:=\{v_1,\ldots, v_{m}\}$ to denote a set of $m$ summation indices ranging from $1$ to $N$ which is split into two disjoint subsets $\mathcal{I}_{m_1}^{(1)}$ and $\mathcal{I}_{m_2}^{(2)}$ with $m_1+m_2=m$. 
For any $n_1 \in \N$, we use $\prod^{n_1}_{i=1} G_{x_{i} y_{i}}$ to denote a product of $n_1$ (not necessarily off-diagonal) Green function entries, where each row and column index $x_i, y_i$ represents an element in the first subset $\mathcal{I}_{m_1}^{(1)}$. We also use $\prod_{l=1}^{n_2} (G_{w_{l} w_{l}}-m_{sc})$ to denote a product of shifted diagonal Green function entries where each $w_l$ represents a different element in $\mathcal{I}_{m_1}^{(2)}$. In particular, we have $n_2=m_2$. Then we define
\begin{align}\label{degree_0}
		n:=n_1+n_2, \qquad d:=\#\{ 1 \leq i\leq n_1: x_i \neq y_i\}+n_2,
\end{align}
where the number $n$ is the total number of (shifted) Green function entries in the product, and the number $d~(\leq n)$ is the number of off-diagonal Green function entries plus the number of shifted diagonal Green function entries, which is also referred to the degree of a term in (\ref{product}) below. Then we define an averaged product of (shifted) Green function entries of degree $d$, \ie
\begin{align}\label{product}
		\frac{1}{N^{m}}  &\sum_{v_1, \ldots,v_m=1}^N   c_{v_1, \ldots, v_m}(t,z)  \prod^{n_1}_{i=1} G_{x_{i} y_{i}} (t,z)\prod_{l=1}^{n_2} (G_{w_{l} w_{l}}-m_{sc}) (t,z) \nonumber\\
	=&:\frac{1}{N^{m}}\sum_{\mathcal{I}
	_m} c_{\mathcal{I}_m} \prod^{n_1}_{i=1} G_{x_i y_i} \prod_{l=1}^{n_2} (G_{w_{l} w_{l}}-m_{sc}), \qquad t \in \R^+,  z \in \C^+,
\end{align}
where $c_{\mathcal{I}_m} \equiv c_{v_1, \ldots, v_m}(t,z)$ is a uniformly bounded and deterministic function. A term in (\ref{product}) of degree $d$ is denoted by $Q_d \equiv Q_d(t,z)$ in general. We use $\mathcal{Q}_d \equiv \mathcal{Q}_d(t,z)$ to denote the collection of the terms of the form in (\ref{product}) of degree $d$. 
\end{definition}

For any $Q_d \in \mathcal{Q}_d$, it is clear from the local law in (\ref{G}) that 
\begin{align}\label{localaw_0}
	|Q_d(t,z)| \prec \Psi^d+N^{-1},
\end{align}
uniformly in $z \in {\mathcal S}_{\mathrm{edge}}$ given in (\ref{S_edge}) and $t \geq 0$, where the last error $N^{-1}$ is from the cases when at least two summation indices coincide in the summation.

Next we follow \cite[Definition 4.2]{Schnelli+Xu} to define unmatched terms of the form in (\ref{product}). Due to (\ref{symmetric}) for the real case, we do not distinguish between the row and column indices. The modifications for unmatched terms in the complex case can be found in \cite[Definition 4.2]{Schnelli+Xu}.

\begin{definition}[Unmatched term, unmatched index]\label{unmatch_def}
Let $Q_d$ be a general term of the form in (\ref{product}). For any summation index ${v}_j \in  \mathcal{I}_m$, let $\nn({v}_j)$ be the number of appearances of the index ${v}_j$ as the row or column index in the product of the (shifted) Green function entries. In particular, for any ${v}_j \in \mathcal{I}_{m_1}^{(1)} \subset \mathcal{I}_m$,
	\begin{align}\label{nu_number}
		\nn({v}_j):=
			\#\{ 1 \leq i \leq n_1: x_i \equiv {v}_j \}+\#\{ 1 \leq i \leq n_1: y_i \equiv {v}_j \}.
	\end{align}
	If $\nn(v_j)$ is odd, then we say that the summation index ${v}_j$ is unmatched. Otherwise, $v_j$ is matched.	Moreover, for any summation index $v_j \in \mathcal{I}_{m_2}^{(2)}\subset \mathcal{I}_m$, $\nn(v_j)=2$ hence $v_j$ is matched .

	 If there exists at least one (equivalently two) unmatched summation index in $\mathcal{I}_{m_1}^{(1)}$, then we say $Q_d$ is an unmatched term, denoted by $Q^o_d$ in general. The collection of unmatched terms of the form in (\ref{form}) with degree $d$ is denoted by $\mathcal{Q}_d^o \subset \mathcal{Q}_d$. 
\end{definition}

Following \cite[Proposition 4.3]{Schnelli+Xu}, the proof of the proposition below is presented in the appendix.
\begin{proposition}\label{unmatch_lemma}
	Consider an unmatched term $Q^o_d \in \mathcal{Q}_d^o$ of the form in (\ref{form}) with fixed $n \in \N$ given in (\ref{degree_0}). Then we have
	\begin{align}\label{unmatch_lemma equation}
		|\E[Q^o_d(t,z)]|=O_{\prec}\big(N^{-1}\big)\,,
	\end{align}
	uniformly in $t \in \R^+$ and $z \in {\mathcal S}_{\mathrm{edge}}$ given in (\ref{S_edge}). 
\end{proposition}

Note that the third order terms in (\ref{fourth_term}) are unmatched terms of the form in (\ref{product}) with a factor $\sqrt{N}$, since $\nn(a)=\nn(b)=3$. Therefore using Proposition \ref{unmatch_lemma}, these third order terms can be bounded by
\begin{align}\label{third_estimate}
	|K_3|=O_{\prec}(N^{-1/2}).
\end{align}
Combining (\ref{fourth_estimate}) and (\ref{third_estimate}) with (\ref{sde_im}), we have
\begin{align}\label{final}
	\Big|\frac{\dd }{\dd t}\E [ \Im m_N(t,z)]\Big| =\ee^{-t} \Big( O_{\prec}\big(\frac{\E[\Im m_N(t,z)]}{N\eta}\big)+O_{\prec}\big(N^{-1/2}\big) \Big)\,,
\end{align}
uniformly in $z \in {\mathcal S}_{\mathrm{edge}}$ and $t\in \R^+$. This proves Proposition \ref{lemma_step_2}.
\end{proof}

\section{Proof of Theorem \ref{green_comparison}}\label{sec:proof}
In this section, we extend the proof of Proposition \ref{GCT_mn} to prove Theorem \ref{green_comparison} for a general function~$F$. The proof will follow the same strategy outlined in Subsection \ref{sec:strategy} and we will address only the modifications needed for a general $F$. More precisely, we divide the proof into two steps:
\begin{enumerate}
	\item	Given the matrix interpolating flow $H^{(1)}$ in (\ref{sum_2}), we iteratively use analogous expansions in Section \ref{sec:step1} for products of Green function entries of the form in (\ref{form}) with a general $F$, in combination with the spectral property of the variance matrix $S$ in (\ref{S_spectrum}) and the improved estimate of $\E[\Im m^{(1)}_N(t,z)]$ in Proposition \ref{GCT_mn_1} to prove that	
	\begin{align}\label{green_difference_1_final}
		\Big| \big(\E^{W^{S}}-\E^{\mathrm{G \beta E}}\big)  \Big[ F\Big( N \int_{E_1}^{E_2} \Im m_N(y+\ii \eta) \dd y\Big) \Big]\Big| =O_\prec(N^{-1/3}).
	\end{align}
	\item	Given the matrix interpolating flow $H^{(2)}$ in (\ref{sum_1}), we adapt the expansions of Green function entries in Section \ref{sec:step2} to estimate the third order (unmatched) terms with a general $F$. Moreover, we use the generalized Ward identity in~(\ref{ward_1})-(\ref{ward_2}) and the improved estimate of $\E[\Im m^{(2)}_N(t,z)]$ in Proposition \ref{GCT_mn_2} to bound the fourth order (matched) terms with a general $F$, \ie
	\begin{align}\label{green_difference_2_final}
		\Big| \big(\E^{H}-\E^{W^{S}}\big)  \Big[ F\Big( N \int_{E_1}^{E_2} \Im m_N(y+\ii \eta) \dd y\Big) \Big]\Big| =O_\prec(N^{-1/3}).
	\end{align}
\end{enumerate}
Combining (\ref{green_difference_1_final}) and (\ref{green_difference_2_final}), we hence finish the proof of Theorem \ref{green_comparison}.

\subsection{Proof of the estimate in (\ref{green_difference_1_final})}\label{subsec:step1}

	In the first step, we consider the modified matrix interpolating flow $H^{(1)}$ in (\ref{sum_2}). For any $E_1,E_2$ such that $2-C_1 N^{-\frac{2}{3}}\leq E_1<E_2\leq 2+C_2  N^{-2/3+\epsilon}$ and $N^{-1+\epsilon} \leq \eta \leq N^{-2/3-\epsilon}$, we define 
	\begin{align}\label{X_1}
			\X^{(1)} = \X^{(1)}(t):=N \int_{E_1}^{E_2}\Im m_N^{(1)}(t,x+\ii \eta) \dd x, \qquad t\in \R^+.
		\end{align}
	Taking the time derivative of $\E[F(\X^{(1)})]$ as the analogue of (\ref{intial_step}), we have 
	\begin{align}\label{sde_im_F_0}
		\frac{\dd}{\dd t}\E[F(\X^{(1)})]=&-\sum_{a, b=1}^{N} \E\Big[F'(\X^{(1)}) \int_{E_1}^{E_2} \frac{\dd h^{(1)}_{ab}}{\dd t} \Im \big( \sum_{v=1}^N G^{(1)}_{va} G^{(1)}_{bv} (t,x+\ii \eta)\big)  \dd x \Big]\nonumber\\
		=&-\sum_{a, b=1}^{N} \E\Big[F'(\X^{(1)})  \frac{\dd h^{(1)}_{ab}}{\dd t} \Dim \big(  G^{(1)}_{ba} \big)  \Big]
	\end{align}
	where we used that $(G^2)(z)=\frac{\dd G(z)}{\dd z}$, and introduced the following short hand notation $\Dim$: For any function $P(t,z): \R^+ \times \C^+ \rightarrow \C^+$, we define 
	\begin{align}\label{dim}
		 \Dim P :=\Im P(t,E_2+\ii \eta)-\Im P(t,E_1+\ii \eta).
	\end{align} 
In the following, we will omit the superscript $(1)$ of $G^{(1)}$ and $\X^{(1)}$ for notational simplicity. Note that from (\ref{dH}), we have a new differentiation rule for a general $F$, \ie
\begin{align}\label{dH_F}
	\frac{\partial \X}{\partial h_{ab}}
	=&-\frac{2}{1+\delta_{ab}} \Im \int_{E_1}^{E_2} (G^2)_{ba} (t,x+\ii\eta) \dd x=-\frac{2}{1+\delta_{ab}}\Dim G_{ba}.
\end{align}
Performing cumulant expansions as in \eqref{sde_im_2}, we obtain from (\ref{sde_im_F_0}) that
	\begin{align}\label{sde_im_F_2}
		\frac{\dd}{\dd t}\E[F(\X^{(1)})]=&-\frac{\ee^{-t}}{2}\sum_{a,b} \big( S_{ab}-\frac{1}{N}\big)(1+\delta_{ab})\E \Big[ \frac{\partial F'(\X) \Dim G_{ba}}{\partial h_{ab}}\Big]\nonumber\\
		=&\ee^{-t}\sum_{a,b}  T_{ab} \E \Big[ F'(\X) \Dim \big(G_{aa} G_{bb} \big) \Big]+\ee^{-t}\sum_{a,b}  T_{ab} \E \Big[ F'(\X) \Dim \big(G_{ab} G_{ba}\big) \Big]\nonumber\\
		&+\ee^{-t}\sum_{a,b}  T_{ab} \E \Big[ F''(\X) \Dim G_{ab} \Dim G_{ba} \Big],
	\end{align}
	where we used (\ref{dH}), (\ref{dH_F}) and that $\partial/\partial h_{ab}$ commutes with $\Dim$ in (\ref{dim}). Then it suffices to prove
\begin{align}\label{green_difference_int}
	\Big|\frac{\dd}{\dd t}\E[F(\X^{(1)}(t))]\Big|=\ee^{-t} O_\prec(N^{-1/3}).
\end{align}
Integrating (\ref{green_difference_int}) over $t \in [0,T_0]$ with $T_0=10\log N$ and using (\ref{approxxxx}), we obtain 
\begin{align}\label{green_difference_2}
	\Big| \big(\E^{W^{\wt S}}-\E^{\mathrm{G \beta E}}\big)  \Big[ F\Big( N \int_{E_1}^{E_2} \Im m_N(y+\ii \eta) \dd y\Big) \Big]\Big| =O_\prec(N^{-1/3}).
\end{align}
Similar to Lemma \ref{lemma_third_step}, the slight modifications on the variance profile matrix $S$ in (\ref{tilde_S}) will not effect the estimate in (\ref{green_difference_2}) up to an error $O_\prec(N^{-1/3})$, and we hence conclude with (\ref{green_difference_1_final}).

The rest of this subsection is devoted to proving (\ref{green_difference_int}). Since the matrix $T$ satisfies the summation property in (\ref{prop_T}), the first term on the right side of (\ref{sde_im_F_2}) is then given by
$$\sum_{a,b}  T_{ab}\E \Big[ F'(\X) \Dim \big(G_{aa} G_{bb}\big) \Big]=\sum_{a,b}  T_{ab}\E \Big[ F'(\X) \Dim \big((G_{aa}-m_{sc}) (G_{bb}-m_{sc})\big) \Big].$$
Thus all the second order terms in (\ref{sde_im_F_2}) can be written as terms of the form in (\ref{form}) with additional derivatives of $F$ and the abbreviations $\Dim$ in front of the (shifted) Green function entries. The naive size of these terms is $O_\prec(N\Psi^2)$ and we next use iterative expansions and the estimate in (\ref{img_wigner_1}) to improve the upper bound to $O_\prec(N^{-1/3})$. 

To study these terms in general, we introduce the following abstract form analogues to (\ref{form}) for a general function $F$.
\begin{definition}[General function $F$, \cf Definition \ref{def_form}]\label{def_form_F}
	For any $d \in \N$ with $d\geq 2$, we use $\mathcal{J}_d:=(j_1,\ldots, j_{d})$ to denote a set of $d$ ordered summation indices ranging from $1$ to $N$. For any $2\leq m_0 \leq d$, the first $m_0$ summation indices $j_1,\ldots,j_{m_0}$ have the non-uniform weights from $\prod_{p=1}^{m_0-1}(T^{k_p})_{j_{p}j_{p+1}}$, and each of the remaining $d-m_0$ indices $j_{m_0+1},\ldots,j_{d}$ has a uniform weight $N^{-1}$ in the summation. Then we consider the abstract form of $d$ (shifted) Green function entries,
	\begin{align}\label{form_F}
		\frac{1}{N^{d-m_0}} \sum_{j_1, \ldots, j_{d}} \prod_{p=1}^{m_0-1} (T^{k_p})_{j_{p}j_{p+1}} \E \Big[F^{(\alpha)}(\X)  \prod_{q=1}^{\alpha}    \Dim\Big( (m_{sc})^{c_q} \prod_{i=1}^{m^{(q)}_1}G_{x^{(q)}_i y^{(q)}_i} \prod_{l=1}^{m^{(q)}_2} \big(G_{w^{(q)}_lw^{(q)}_l}-m_{sc}\big) \Big) \Big], 
	\end{align}
	with $\{k_p\},\{c_q\} \in \N$, where $F^{(\alpha)}$ is the $\alpha$-th derivative of the smooth function $F$,  each row and column index of the (shifted) Green function entries $x^{(q)}_i$, $y^{(q)}_i$ and $w^{(q)}_l$ represents an element in $\mathcal{J}_d$ with $x^{(q)}_i \neq y^{(q)}_i$, and each element in $\mathcal{J}_d$ appears exactly twice as the row or column index. In particular, $d=\sum_{q=1}^{\alpha}(m_1^{(q)}+m_2^{(q)})$. A term of the form in (\ref{form_F}) with degree $d$ is in general denoted by $\wt P_{d} \equiv \wt P_{d}(t,z)$.  The collection of the terms in (\ref{form_F}) with degree $d$ is denoted by $\wt{\mathcal{P}}_{d} \equiv \wt{\mathcal{P}}_{d}(t,z)$.
\end{definition}
We remark that the deterministic functions $(m_{sc})^{c_q}$ in (\ref{form_F}) cannot be moved outside due to $\Dim$ defined in (\ref{dim}). Since the derivatives of $F$ are uniformly bounded and that $|m_{sc}|\leq 1$, the naive sizes in (\ref{initial}) and (\ref{large_K_term}) also hold true for any term $\wt P_d$ of the form in (\ref{form_F}).

Given any term $\wt P_d$ in (\ref{form_F}), using the new differentiation rule in (\ref{dH_F}) to compute $\partial F^{(\alpha)}(\X)/\partial h_{ab}$, we extend the expansions in (\ref{expand_case_diagonal}) and (\ref{expand_case_offdiagonal}) to the cases with a general $F$. 

\begin{enumerate}
	\item[(a)] If the special index $j_1$ appears in a (shifted) diagonal Green function entry, say $w^{(1)}_1 \equiv j_1$, then we have \cf (\ref{cf_expand}),
	\begin{align}\label{cf_expand_F}
		\E[\wt P_d]=&(1-\ee^{-t})\E[\wt P_d(k_1 \rightarrow k_1+1, c_1 \rightarrow c_1+2)]+\sum_{\wt P_{d_1} \in \wt{\mathcal{P}}_{d_1}, d_1 = d+1} \E[\wt P_{d_1}],
	\end{align}
where the leading term of degree $d$ is obtained from the original term $\wt P_d$ by increasing the power of the matrix $T^{k_1}$ to $k_1+1$ and increasing the power of the function $(m_{sc})^{c_1}$ to $c_1+2$,  the second group of terms contains at most $4(d+1)$ terms of the form in (\ref{form_F}) with higher degrees $d+1$ and possible factors $1-\ee^{-t}$. Iterating (\ref{cf_expand_F}) for $K-k_1$ times with $K$ given in (\ref{K_choice}), we obtain the analogue of (\ref{expand_case_diagonal}), \ie
	\begin{align}\label{expand_case_diagonal_F}
		\E[\wt P_d]=&\sum_{\wt P_{d'} \in \wt{\mathcal{P}}_{d'}, d'= d+1} \E[\wt P_{d'}]+O_\prec(N^{-8}),
	\end{align}
	where the first group of terms contains at most $4K(d+1)$ terms of the form in (\ref{form_F}) with higher degrees $d+1$, and each term comes with a possible factor $1-\ee^{-t}$. 
	
	\item[(b1)] If the special index $j_1$ appears in two different off-diagonal Green function entries with the same superscript $(q)$, say $x^{(1)}_1,x^{(1)}_2 \equiv j_1$ and $y^{(1)}_1,y^{(1)}_2 \not\equiv j_1$, then we have \cf (\ref{expand_case_2}),
	\begin{align}\label{expand_case_2_F}
		\E[\wt P_d]=&(1-\ee^{-t}) \E[\wt P_d(k_1 \rightarrow k_1+1,c_1\rightarrow c_1+2)]\big)+\sum_{\substack{\wt P_{d_1} \in \wt{\mathcal{P}}_{d_1},\\d_1 = d+1}}  \E[\wt P_{d_1}]+O_{\prec}(N^{-1/3}),
	\end{align}
where the last error $O_\prec(N^{-1/3})$ is obtained from the cases with index coincidence (see (\ref{im_estiamte_2})-(\ref{im_estiamte_4})) using the improved estimate of $\E[\Im m^{(1)}_N(t,z)]$ in (\ref{img_wigner_1}), the second group of terms in (\ref{expand_case_2_F}) contains at most $4(d+1)$ terms of the form in (\ref{form_F}) with possible factors $1-\ee^{-t}$.

	\item[(b2)] If the special index $j_1$ appears in two off-diagonal Green function entries with different superscripts $(q)$, say $x^{(1)}_1,x^{(2)}_1 \equiv j_1$ and $y^{(1)}_1,y^{(2)}_2 \not\equiv j_1$, then we have similar to (\ref{expand_case_2_F}),
	\begin{align}\label{expand_case_2_FF}
	\E[\wt P_d]=&(1-\ee^{-t}) \E[\wt P_d(k_1 \rightarrow k_1+1,c_1\rightarrow c_1+1,c_2 \rightarrow c_2+1)]\big)\nonumber\\
	&\qquad \quad+\sum_{{\wt P_{d_1} \in \wt{\mathcal{P}}_{d_1},d_1 = d+1}}  \E[\wt P_{d_1}]+O_{\prec}(N^{-1/3}).
\end{align}
	Iterating either of the expansions in (\ref{expand_case_2_F}) or (\ref{expand_case_2_FF}) for $K-k_1$ times with $K$ given in (\ref{K_choice}), we obtain the analogue of (\ref{expand_case_offdiagonal}), \ie
	\begin{align}\label{expand_case_offdiagonal_F}
		\E[\wt P_d]	=&\sum_{\wt P_{d''} \in \wt{\mathcal{P}}_{d''}, d''= d+1}\E[\wt P_{d''}]+O_{\prec}(N^{-8})+O_{\prec}(K N^{-1/3}),
	\end{align}
	where the first group of terms contains at most $4K(d+1)$ terms of the form in (\ref{form_F}) with higher degrees $d+1$ and possible factors $1-\ee^{-t}$, and the last error $O_\prec(KN^{-1/3})$ is from the cases with index coincidences using (\ref{img_wigner_1}).
\end{enumerate}

We omit the proof details since they are quite similar to Subsection \ref{subsec:novel_expansion} using additionally the new differentiation rule in (\ref{dH_F}). Now we go back to (\ref{sde_im_F_2}), whose right side consists of three terms of the form in (\ref{form_F}) with degree two. For a general term of the form in (\ref{form_F}) with degree at least two, denoted by $\wt P_d$, iterating the expansions in (\ref{expand_case_diagonal_F}) and (\ref{expand_case_offdiagonal_F}) for $D \geq \frac{4}{\epsilon}$ times as in (\ref{expand_mechanism})-(\ref{four}), we have
\begin{align}
		\E[\wt P_d]=&O_{\prec}\big( (8KD)^{D} (N \Psi^D+N^{-8})\big)+O_{\prec}\big((8KD)^{D}N^{-1/3}\big)=O_{\prec}(N^{-1/3}).
\end{align}
We hence have proved the estimate in (\ref{green_difference_int}) and conclude with (\ref{green_difference_1_final}).

\subsection{Proof of the estimate in (\ref{green_difference_2_final})}\label{subsec:step2}

In the second step, recalling the matrix interpolating flow (\ref{sum_1}), we define as in (\ref{X_1})
\begin{align}
	\X^{(2)} = \X^{(2)}(t):=N \int_{E_1}^{E_2}\Im m_N^{(2)}(t,x+\ii \eta) \dd x.
\end{align}
Taking the time derivative of $\E[F(\X^{(2)}(t))]$, using the differentiation rules in (\ref{dH}) and (\ref{dH_F}), we have the analogue of (\ref{sde_im}), \ie
\begin{align}\label{sde_im_F}
	\frac{\dd}{\dd t}\E[F(\X^{(2)}(t))]=&-\frac{\ee^{-t}}{2}\sum_{k+1=3}^{4} \frac{1}{k!} \frac{s^{(k+1)}_{ab}(t)}{N^{\frac{k+1}{2}}}  \sum_{a,b} \E \Big[\frac{\partial^k F'(\X) \Dim G_{ba}}{\partial h_{ab}^k}\Big]+O_{\prec}\Big(\frac{1}{\sqrt{N}}\Big),
\end{align}
with $s^{(k+1)}_{ab}(t)$ given in (\ref{s_ab_t}), and we set $G= G^{(2)}(t,z)$, $\X = \X^{(2)}$ for notational simplicity. 

Using the differentiation rules in (\ref{dH}) and (\ref{dH_F}), and that $\partial/\partial h_{ab}$ commutes with $\Dim$ in (\ref{dim}), the third order terms with $k+1=3$ in (\ref{sde_im_F}) can be written as averaged products of Green function entries of the form in (\ref{product}) up to a factor $\sqrt{N}$, with additional derivatives of $F$ and the abbreviations $\Dim$ in front of the Green function entries, \eg
\begin{align}
	&\frac{\sqrt{N}}{N^{2}} \sum_{a,b} s_{ab}^{(3)}(t)\E \big[F'(\X) \Dim (G_{aa} G_{aa}G_{ab})\big], \qquad \frac{\sqrt{N}}{N^{2}} \sum_{a,b} s_{ab}^{(3)}(t)\E \big[F''(\X) \Dim(G_{ab}) \Dim (G_{aa} G_{aa})\big];\nonumber\\
	&\qquad \qquad \qquad\qquad\frac{\sqrt{N}}{N^{2}} \sum_{a,b} s_{ab}^{(3)}(t)\E \big[F^{(3)}(\X) \Dim(G_{ab}) \Dim(G_{ab})\Dim(G_{ab})\big].
\end{align}
Since the index $a$ and $b$ appear three times in the product of Green function entries as the row or column index, these third order terms are unmatched terms; see Definition \ref{unmatch_def}. The statement of Proposition~\ref{unmatch_lemma} still holds true for such a general form and the proof is quite similar using additionally the new differentiation rule in (\ref{dH_F}) for a general $F$. Hence, the third order unmatched terms in (\ref{sde_im_F}) with additional factor $\sqrt{N}$ can be bounded by 
$$\sum_{k+1=3} \mbox{on r.h.s. of (\ref{sde_im_F})}=O_\prec(N^{-1/2}).$$

By direct computations using (\ref{dH}) and (\ref{dH_F}), the fourth order terms with $k+1=4$ in (\ref{sde_im_F}) are also averaged products of Green function entries with additional derivatives of $F$ and $\Dim$ in front, \eg
\begin{align}\label{some_term_F}
	&\frac{1}{N^{2}} \sum_{a,b} s_{ab}^{(4)}(t)\E \big[F'(\X) \Dim (G_{aa} G_{aa}G_{bb}G_{bb})\big]; 
	\nonumber\\
	&\frac{1}{N^{2}} \sum_{a,b} s_{ab}^{(4)}(t)\E \big[F''(\X)\Dim (G_{aa} G_{bb} ) \Dim (G_{aa} G_{bb} )\big];\nonumber\\
	&\frac{1}{N^{2}} \sum_{a,b} s_{ab}^{(4)}(t)\E \big[F^{(3)}(\X)\Dim (G_{ab} ) \Dim (G_{ab} ) \Dim (G_{aa} G_{bb}) \big]; \nonumber\\
	&\frac{1}{N^{2}} \sum_{a,b} s_{ab}^{(4)}(t)\E \big[F^{(4)}(\X)\Dim (G_{ab} ) \Dim (G_{ab} ) \Dim (G_{ab} ) \Dim (G_{ab} )\big].
\end{align}
Using the definition of $\Dim$ in (\ref{dim}), the estimate of $\Im G_{aa}$ in (\ref{ward_1}), the generalized Ward identity in (\ref{ward_2}), and that all derivatives of $F$ are uniformly bounded, these fourth order terms can be bounded by 
$$\sum_{k+1=4} \mbox{on r.h.s. of (\ref{sde_im_F})}=O_\prec\left(\E[\Im m_{N}(t,E_1+\ii \eta)]+\E[\Im m_{N}(t,E_2+\ii \eta)]\right).$$
Combining with the estimate of $\E[\Im m^{(2)}_{N}(t,z)]$ in (\ref{img_wigner_2}), we obtain
\begin{align}\label{green_difference_int_2}
	\Big|\frac{\dd}{\dd t}\E[F(\X^{(1)})]\Big|=\ee^{-t} O_{\prec}(N^{-1/3}),
\end{align}
where the error is proportional to the fourth order cumulants of the normalized entries of the generalized Wigner matrix~$H$ in (\ref{s_ab_t}). Integrating (\ref{green_difference_int_2}) over $t \in [0,T_0]$ with $T_0=10\log N$ and using (\ref{approxxxx}), we finish the proof of (\ref{green_difference_2_final}).

\section{Appendix}

\subsection{Proof of Lemma \ref{lemma_third_step}}
We consider the interpolating flow between two independent Gaussian matrices with variance profile matrix $\wt S$ and $S$ respectively, \ie
$$\wt H(t):=\mathrm{e}^{-\frac{t}{2}}W^{\wt S} +\sqrt{1-\mathrm{e}^{-t}} W^{S}, \qquad (\wt S)_{ab} =(S)_{ab} (1+\delta_{ab}).$$
We define the Green function of $\wt H(t)$ by $\wt G=\wt G(t,z)$ and its normalized trace by $\wt m_N(t,z)$. Taking the time derivative of $\E[\Im \wt m_N(t,z)]$ and performing the cumulant expansions as in (\ref{intial_step})-(\ref{sde_im_2}), we obtain
\begin{align}
	\frac{\dd}{\dd t}\E[\Im \wt m_N(t,z)]=&-\frac{1}{N}\sum_{v=1}^N\sum_{a, b} \E\Big[\Im \big( \wt G_{va} \wt G_{bv}\frac{\dd \wt h_{ab}(t)}{\dd t} \big)\Big]\nonumber\\
	=&-\frac{\ee^{-t}}{2} \frac{1}{N}\sum_{v=1}^N\sum_{a , b} \big(S_{ab}- \wt S_{ab}\big) \E \Big[\Im \frac{\partial \wt G_{va} \wt G_{bv}}{\partial h_{ab}}\Big]\nonumber\\
	=&\ee^{-t}\frac{1}{N}\sum_{v,a} S_{aa} \E \Big[\Im \big( \wt G_{va} \wt G_{aa} \wt G_{av}\big)\Big].
\end{align}
We remark that the local law in (\ref{G}) and the generalized Ward identity in (\ref{ward_1})-(\ref{ward_2}) also hold for the Green function $\wt G$. Using that $S_{aa}=O(N^{-1})$, we obtain that
\begin{align}
	\big|\frac{\dd}{\dd t}\E[\Im \wt m_N(t,z)]\big| \prec C \ee^{-t} \frac{\E[\Im \wt m_N(t,z)]}{N \eta}.
\end{align}
Using Gr\"onwall's inequality, we find that, for any $t\in \R^+$ and $z \in {\mathcal S}_{\mathrm{edge}}$ in (\ref{S_edge}), 
\begin{align}
	\E[\Im \wt m_N(t,z)] \leq C	\E[\Im \wt m_N(0,z)]=C	\E^{\wt S}[\Im \wt m_N(z)].
\end{align}
We hence finish the proof of Lemma \ref{lemma_third_step}.

\subsection{Proof of Proposition \ref{unmatch_lemma}}

Following \cite[Proposition 4.3]{Schnelli+Xu}, in order to prove Proposition \ref{unmatch_lemma}, we first introduce the expansion mechanism of a given term of the form in (\ref{product}) for generalized Wigner matrices with inhomogeneous variances.  

Given an unmatched term $Q_d^o$ in (\ref{product}), let the index $a \in \mathcal{I}_{m_1}^{(1)}$ to be an unmatched index without loss of generality. The expansion via the index $a$ is split into the following two cases:

{\bf Case 1:} 
If there exists a diagonal factor $G_{aa}$ in the first product of Green function entries $\prod^{n_1}_{i=1} G_{x_i y_i}$, then we have
\begin{align}\label{case1}
	\E[Q^o_d]=\E[Q^o_d\big( G_{aa} \rightarrow m_{sc}\big)]+ \sum_{ \substack{Q^o_{d''} \in \mathcal{Q}^o_{d''}  \\ d'' \geq d+1 }} \E[Q^o_{d''}]+\frac{1}{\sqrt{N}} \sum_{Q^o_{d'} \in \mathcal{Q}^o_{d'}; d' \geq d} \E [Q^o_{d'}]+O_{\prec}\big(N^{-1}\big),
\end{align}
where we replaced the diagonal Green function entry $G_{aa}$ with $m_{sc}$ for the leading term which remains unmatched. The second group of terms contains at most $2n$ terms of the form in (\ref{product}) with higher degrees, denoted by $Q^o_{d''}$ in general, and the third group of terms contains at most $4(n+1)^2$ terms denoted by $Q^o_{d'}$ with an additional factor $\frac{1}{\sqrt{N}}$.

{\bf Case 2:}
If there is no diagonal factor $G_{aa}$ in the first product of Green function entries $\prod^{n_1}_{i=1} G_{x_i y_i}$, assuming that $x_1 \equiv a$, $y_1\not \equiv a$ without loss of generality, then we have
\begin{align}\label{case2}
	\E[Q^o_d]=&m^2_{sc}\sum_{\substack{2\leq i \leq n_1\\x_i \equiv a,y_i \not\equiv a}} \E \Big[ Q^o_{d} \big( x_1 ,x_i \equiv a \rightarrow j \big)\Big]+m^2_{sc}\sum_{\substack{2\leq i \leq n_1\\ x_i\not\equiv a,y_i \equiv a}} \E \Big[ Q^o_{d} \big( x_1, y_i \equiv a\rightarrow j \big)\Big]\nonumber\\
	&+\sum_{ \substack{Q^o_{d''} \in \mathcal{Q}^o_{d''}  \\ d'' \geq d+1} } \E[Q^o_{d''}]+\frac{1}{\sqrt{N}} \sum_{\substack{Q^o_{d'} \in \mathcal{Q}^o_{d'}\\ d' \geq d}} \E [Q^o_{d'}]+O_{\prec}(N^{-1}),
\end{align}
where we have replaced a pair of the index $a$ from two distinct off-diagonal Green function entries with the fresh index $j$ for the leading terms, and the index $a$ remains unmatched. The first group of terms on the second line contains at most $2n$ terms of the form in (\ref{product}) with higher degrees, and the second group of terms contains at most $4(n+1)^2$ terms of the form in (\ref{product}) with an additional factor $\frac{1}{\sqrt{N}}$.

Next we will sketch the proof of Proposition \ref{unmatch_lemma} by iteratively using these two type of expansions above, and similar arguments can be found in \cite{Schnelli+Xu,SXsample}. Given an unmatched term $Q_d^o$ of the form in (\ref{product}) with fixed $n$ given in (\ref{degree_0}), since the index $a$ is unmatched, the number of appearances of the index $a$, \ie $\nn(a)$ defined in (\ref{nu_number}), is odd and $\nn(a)\leq n$. 

If $\nn(a)=1$, since there is no $G_{aa}$ factor, we perform the expansion in (\ref{case2}). There will be no leading term of degree $d$ in the first line of (\ref{case2}) and we obtain finitely many unmatched terms with higher degrees at least $d+1$. Otherwise, if $\nn(a) \geq 3$, performing either the expansion in (\ref{case1}) or (\ref{case2}), the number of appearances of the index $a$ in each resulting leading term on the right side has been decreased by two. Thanks to $a$ still being unmatched, we can further expand these leading terms via the index $a$ until the number of appearances of $a$ is reduced to one. Then we go back to the previous case with $\nn(a)=1$.

In this way, we have expanded the unmatched term $\E[Q_d^{o}]$ into finitely many unmatched terms of degrees at least $d+1$ with improved estimate from (\ref{localaw_0}), \ie
\begin{align}\label{third_unmatch2}
	\E[Q_d^o(t,z)]=\sum_{\substack{ Q^o_{d'} \in \mathcal{Q}^o_{d'} \\ d' \geq d+1}}  \E[Q^o_{d'}(t,z)]+\frac{1}{\sqrt{N}} \sum_{\substack{Q^o_{d''} \in \mathcal{Q}^o_{d''}\\ d'' \geq d}} \E [Q^o_{d''}(t,z)]+O_{\prec}\big(N^{-1}\big)\,,
\end{align}
where the number of unmatched terms on the right side above is bounded by $(Cn)^{cn}$, and the number of the Green function entries in  each term is bounded by $Cn$ for some numerical constants~$C,c>0$.

Iterating the expansion process in (\ref{third_unmatch2}) for $D-d$ times with $D>d$ sufficiently large fixed later, the first group of terms on the right side of (\ref{third_unmatch2}) contains at most $\big((C^Dn)^{c^D n} \big)^D$ terms with degrees at least $D$. Similarly the second group of terms with a factor $\frac{1}{\sqrt{N}}$ on the right side of (\ref{third_unmatch2}) contains at most $\big((C^Dn)^{c^D n} \big)^D$ terms with degrees at least $D-1$. We hence obtain from the naive estimate in~(\ref{localaw_0}) that
\begin{align}\label{third_unmatch3}
	|\E[Q_d^o(t,z)]|=O_{\prec}\Big(\Psi^D+\frac{\Psi^{D-1}}{\sqrt{N}}+\frac{1}{N}\Big)=O_{\prec}(\Psi^D+N^{-1})\,,
\end{align}
where the error $O_\prec(N^{-1})$ is from the cases with index coincidences. Choosing $D>\frac{1}{\epsilon}$ sufficiently large depending only on $\epsilon>0$, we finish the proof of Proposition \ref{unmatch_lemma}.

The rest of this subsection is devoted to proving the expansions in (\ref{case1}) and (\ref{case2}). We start with proving the first expansion in (\ref{case1}). We may assume $G_{x_1y_1}=G_{aa}$ without loss of generality. Using that $-\frac{1}{m_{sc}}=z+m_{sc}$ and the definition of resolvent, we have
\begin{align}\label{step1}
	-\frac{\E[Q^o_d]}{m_{sc}}
	=&\E \Big[ \frac{1}{N^{m}}\sum_{\mathcal{I}_m} c_{\mathcal{I}_m} \big( \sum_{j} h_{aj} G_{ja}-1\big) \prod_{i=2}^{n_1} G_{x_i y_i}\prod_{l=1}^{n_2} (G_{w_{l} w_{l}}-m_{sc}) \Big]+m_{sc}\E[Q_d]\nonumber\\
	=&-\E \Big[ \frac{1}{N^{m}}\sum_{\mathcal{I}_m} c_{\mathcal{I}_m}  \prod_{i=2}^{n_1} G_{x_i y_i}\prod_{l=1}^{n_2} (G_{w_{l} w_{l}}-m_{sc}) \Big]+m_{sc}\E[ Q_d]\nonumber\\
	&+\E \Big[ \frac{1}{N^{m}} \sum_{\mathcal{I}_m} c_{\mathcal{I}_m}\sum_{j=1}^N S_{aj}  \frac{ \partial G_{ja}  \prod_{i=2}^{n_1} G_{x_i y_i} \prod_{l=1}^{n_2} (G_{w_{l} w_{l}}-m_{sc}) }{\partial h_{aj}}\Big]\nonumber\\
	&+\E \Big[ \frac{1}{N^{m+\frac{3}{2}}} \sum_{\mathcal{I}_m} c_{\mathcal{I}_m}\sum_{j=1}^N \frac{c^{(3)}(\sqrt{N}h_{aj})}{2!}  \frac{ \partial^2 G_{ja}  \prod_{i=2}^{n_1} G_{x_i y_i} \prod_{l=1}^{n_2} (G_{w_{l} w_{l}}-m_{sc}) }{\partial h^2_{aj}}\Big]+O_{\prec}(N^{-1}),
\end{align}
where $c^{(3)}(\sqrt{N}h_{aj})$ is the third cumulant of the normalized entry $\sqrt{N}h_{aj}$ which is of the constant order, and the last error stems from the truncating the cumulant expansions at the third order.

Using (\ref{dH}), the resulting terms in the last line above are of the form in (\ref{product}) up to a factor $\frac{1}{\sqrt{N}}$, with a new summation index set $\mathcal{I}_{m+1}=\mathcal{I}_m \cup \{j\}$ and two more Green function entries in the product, \ie $n'=n+2$. Since $j$ is a fresh index with $\nn(j)=3$, by direct computations using (\ref{dH}), they are $4n(n+1)$ unmatched terms of degrees at least $d$. In general, we denote all of these terms together by
\begin{align}\label{higher}
	\frac{1}{\sqrt{N}} \sum_{Q^o_{d'} \in \mathcal{Q}^o_{d'}; d' \geq d} \E [Q^o_{d'}].
\end{align}

Similarly, the resulting terms in the second last line of (\ref{step1}) are of the form in (\ref{product}) with $\mathcal{I}_{m+1}=\mathcal{I}_m \cup \{j\}$ and $n'=n+1$. It is straightforward to check that $\nn(j)=2$ for the fresh index $j$ and $\nn(v_j)$ remains the same for any original $v_j$ in $\mathcal{I}_m$ (including the index $a$) . Hence the index $a$ remains unmatched. Using the differentiation rule in (\ref{dH}), we obtain $2n-1$ unmatched terms with higher degrees at least $d+1$, which are denoted in general by
\begin{align}\label{general}
	\sum_{ {Q^o_{d''} \in \mathcal{Q}^o_{d''}, d'' \geq d+1 }} \E[Q^o_{d''}],
\end{align}
except one leading term with degree $d$ from taking $\frac{\partial G_{aj}}{\partial h_{ja}}$, \ie
\begin{align}
	\E \Big[ \frac{1}{N^{m}}\sum_{\mathcal{I}_m} c_{\mathcal{I}_m}\sum_{j=1}^N S_{aj}   G_{jj} G_{aa} \prod_{i=2}^{n_1} G_{x_i y_i} \prod_{l=1}^{n_2} (G_{w_{l} w_{l}}-m_{sc})\Big].
\end{align}
Since $\sum_{j}S_{aj}=1$, there is a cancellation between this leading term and the second term on the right side of (\ref{step1}), \ie $m_{sc}\E[Q_d]$, and we end up with an unmatched term with higher degree $d+1$, \ie
\begin{align}\label{calcel}
	\E \Big[ \frac{1}{N^{m}}\sum_{\mathcal{I}_m,j}c_{\mathcal{I}_m} S_{aj}   (G_{jj}-m_{sc}) G_{aa} \prod_{i=2}^{n_1} G_{x_i y_i} \prod_{l=1}^{n_2} (G_{w_{l} w_{l}}-m_{sc})\Big].
\end{align}
We note that the naive estimate in (\ref{localaw_0}) still holds true and the additional factor $S_{aj}=O(N^{-1})$ is not harmful in the iterative expansions. Therefore, multiplying $-m_{sc}$ on both sides of (\ref{step1}), we have proved  the first expansion in (\ref{case1}).

Next, we continue to prove the second expansion in (\ref{case2}). Assuming there is no $G_{aa}$ factor in $\prod_{i=1}^{n_1} G_{x_iy_i}$, we choose to expand an off-diagonal Green function entry, say $G_{ay_1}$ with $y_1\not\equiv a$. Similarly to (\ref{step1}), we obtain that
\begin{align}\label{step2}
	-\frac{\E[Q^o_d]}{m_{sc}}=&\E \Big[ \frac{1}{N^{m}}\sum_{\mathcal{I}_m} c_{\mathcal{I}_m} \big( \sum_{j} h_{aj} G_{jy_1}-\delta_{ay_1}\big) \prod_{i=2}^{n_1} G_{x_i y_i}\prod_{l=1}^{n_2} (G_{w_{l} w_{l}}-m_{sc}) \Big]+m_{sc}\E[Q_d]\nonumber\\
	=&m_{sc}\E[Q_d]+\E \Big[ \frac{1}{N^{m}} \sum_{\mathcal{I}_m} c_{\mathcal{I}_m}\sum_{j=1}^N S_{aj}  \frac{ \partial G_{jy_1}  \prod_{i=2}^{n_1} G_{x_i y_i} \prod_{l=1}^{n_2} (G_{w_{l} w_{l}}-m_{sc}) }{\partial h_{aj}}\Big]\nonumber\\
	&+\frac{1}{\sqrt{N}} \sum_{Q^o_{d'} \in \mathcal{Q}^o_{d'}; d' \geq d} \E [Q^o_{d'}]+O_{\prec}(N^{-1}),
\end{align}
where the third order terms are obtained similarly as in (\ref{higher}), and the error $O_{\prec}(N^{-1})$ is from truncating the cumulant expansions at the third order and the terms with the index coincidence $\delta_{a y_1}$.

We next estimate the second group of terms on the right side of (\ref{step2}). Since $j$ is a fresh index, the resulting terms have higher degrees at least $d+1$ denoted by (\ref{general}) in general, except the leading terms obtained from acting $\frac{\partial }{\partial h_{aj}}$ on $G_{jy_1}$ or another off-diagonal Green function entry $G_{x_i y_i}~(2\leq i\leq n_1)$ with either $x_i \equiv a$ or $y_i \equiv a$. The leading term corresponding to $\partial G_{jy_1}/\partial h_{aj}$ is given by
$$-\E \Big[ \frac{1}{N^{m}}\sum_{\mathcal{I}_m} c_{\mathcal{I}_m}\sum_{j=1}^N S_{aj}   G_{jj} G_{ay_1} \prod_{i=2}^{n_1} G_{x_i y_i} \prod_{l=1}^{n_2} (G_{w_{l} w_{l}}-m_{sc})\Big],$$
which will be canceled with the first term $m_{sc}\E[Q_d]$ on the right side of (\ref{step2}) as in (\ref{calcel}). For the remaining leading terms, \eg we may assume that $G_{x_2 y_2}=G_{a y_2}$ with $y_2 \not\equiv a$, then the corresponding leading term is given by
\begin{align}\label{example_1}
	-\E \Big[\frac{1}{N^{m} }\sum_{\mathcal{I}_m,j} S_{aj} c_{\mathcal{I}_m}   G_{a a}  G_{j y_1} G_{ jy_2} \prod^{n_1}_{i=3} G_{x_i y_i} \prod_{l=1}^{n_2} (G_{w_{l} w_{l}}-m_{sc}) \Big].
\end{align}
This term has a diagonal factor $G_{aa}$ which can be expanded using the expansion (\ref{case1}) proved in the Case~1. Then we replace one pair of the index $a$ from $G_{a y_1}$ and $G_{a y_2}$ with the fresh index $j$, up to a factor $-m_{sc}$. In general, we write these leading terms for short as
\begin{align}\label{short_temp}
	-m_{sc}\sum_{\substack{2\leq i \leq n_1\\ x_i\equiv a,y_i \not\equiv a}} \E\Big[Q_{d} \big( x_1,x_i \equiv a \rightarrow j \big)\Big]-m_{sc}\sum_{\substack{2\leq i \leq n_1\\x_i \not\equiv a,y_i \equiv a}} \E\Big[Q_{d} \big( x_1,y_i \equiv a\rightarrow j \big)\Big].
\end{align}
Multiplying $-m_{sc}$ on both sides of (\ref{step2}), we have finished the proof of the second expansion in (\ref{case2}).

\medskip

\textit{Acknowledgement: }
This material is based upon work supported by the National Science Foundation under Grant No. DMS-1928930 while K.S. participated in a program hosted by the Mathematical Sciences Research Institute in Berkeley, California, during the Fall 2021 semester.


\begin{thebibliography}{00}


\bibitem{AH} Adhikari, A., Huang, J.: \emph{Dyson Brownian motion for general $\beta$ and potential at the edge},Probab.\ Theory Rel.\ Fields \textbf{178.3} 893-950  (2020).

\bibitem{AEK17} Ajanki, O., Erd\H{o}s, L., Kr\"uger, T.: \emph{Quadratic vector equations on complex upper half-plane} Vol. 261, No. 1261, American Mathematical Society  (2019).

\bibitem{XXX} Alt, J., Erd{\H o}s, L, Kr\"uger, T,  Schr\"oder, D.: \emph{Correlated random matrices: band rigidity and edge universality}, Ann.\ Probab.\ \textbf{48(2)}, 963-1001  (2020).		 



\bibitem{bao} Bao, Z.G., Pan, G.M.,  Zhou, W.: \emph{Universality for the largest eigenvalue of sample covariance matrices with general population}, Ann.\ Stat.\ {\bf 43}(1), 382-421 (2015).


\bibitem{huang_dregular} Bauerschmidt, R., Huang, J., Knowles, A., Yau, H. T.:  \emph{Edge rigidity and universality of random regular graphs of intermediate degree}, Geometric and Functional Analysis, 30(3), 693-769 (2020).


\bibitem{bickel} Bickel, P. J., Sarkar, P.: \emph{Hypothesis testing for automated community detection in networks}, J.\ R.\ Stat.\ Soc.\ Ser.\ B Stat.\ Methodol., \textbf{78(1)}, 253-273  (2016).

\bibitem{Alex+Erdos+Knowles+Yau+Yin}  Bloemendal, A., Erd\H{o}s, L., Knowles, A.,  Yau, H.-T. and Yin, J.: \emph{Isotropic local laws for sample covariance and generalized Wigner matrices}, Electron. J. Probab. {\bf 19} (2014).


\bibitem{Bourgade extreme} Bourgade, P.: \emph{Extreme gaps between eigenvalues of Wigner matrices}, J.\ Eur.\ Math.\ Soc.\ (JEMS) \textbf{24.8}, 2823-2873  (2021).


\bibitem{BEY edge universality} Bourgade, P., Erd{\H o}s, L., Yau, H.-T.:\emph{Edge universality of beta ensembles}, Commun.\ Math.\ Phys.\ \textbf{332.1} 261-353  (2014).


\bibitem{BEYY} Bourgade, P.,  Erd\H{o}s, L.,  Yau, H.-T., Yin, J.: \emph{Fixed energy universality for generalized Wigner matrices}, Comm.\ Pure Appl.\ Math.\ {\bf 69},  (2015).




\bibitem{BoutetdeMonvel} de Monvel, A. Boutet, Khorunzhy, A.: \emph{Asymptotic distribution of smoothed eigenvalue density. II. Wigner random matrices}, Random Oper. Stoch. Equ. {\bf 7}(2), 149-168 (1999).


\bibitem{choup} Choup, L. N.: \emph{Edgeworth expansion of the largest eigenvalue distribution function of GUE and LUE.} International Mathematics Research Notices, 2006(9), 61049-61049 (2006).



\bibitem{Deift} Deift, P., Gioev, D.: \emph{Universality at the edge of the spectrum for unitary, orthogonal, and symplectic ensembles of random matrices}, Comm.\ Pure Appl.\ Math.\ {\bf 60(6)}, 867-910 (2007).	


\bibitem{Deift2} Deift, P., Gioev, D., Kriecherbauer, T., Vanlessen, M.: \emph{Universality for Orthogonal and Symplectic Laguerre-Type Ensembles}, Journal of Statistical Physics {\bf 129}(5-6) (2007).


\bibitem{Ding+Yang} Ding, X.,  Yang, F.: \emph{A necessary and sufficient condition for edge universality at the largest singular values of covariance matrices}, Ann.\ Appl.\ Probab.\ {\bf 28}(3), 1679-1738 (2018).



\bibitem{EYS} Erd\H{o}s, L., Schlein, B., Yau, H. T.: \emph{Universality of random matrices and local relaxation flow}. Invent.\ Math.\ {\bf 185(1)}, 75-119 (2011).


\bibitem{sparse0}  Erd\H{o}s, L., Knowles, A., Yau, H.-T., Yin, J: \emph{Spectral statistics of Erd\H{o}s-R\'enyi Graphs II: Eigenvalue spacing and the extreme eigenvalues}, Commun.\ Math.\ Phys.\ \textbf{314(3)}, 587-640  (2012).



\bibitem{Erdos+Knowles+Yau}  Erd\H{o}s, L., Knowles, A. and  Yau, H.-T.: \emph{Averaging fluctuations in resolvents of random band matrices}, Ann. Henri Poincar\'e {\bf 14} 1837-1926 (2013).	


\bibitem{ErdoesKruegerSchroeder} Erd{\H o}s, L,  Kr\"uger, T., Schr\"oder, D.: {\it Random matrices with slow correlation decay}, Forum of Mathematics Sigma (2019), {\bf 7}(8) (2019).

\bibitem{book}  Erd\H{o}s, L., Yau, H.-T.: {\it A dynamical approach to random matrix theory. Courant Lecture Notes in Mathematics} {\bf 28}. Providence: American Mathematical Society  (2017).

\bibitem{EYY1} Erd\H{o}s, L., Yau, H.-T., Yin, J.: \emph{Bulk Universality for Generalized Wigner Matrices}, Probab.\ Theory Rel.\ Fields \textbf{154(1-2)}, 341-407 (2012).

\bibitem{rigidity} Erd\H{o}s, L, Yau, H.-T., Yin, J.: \emph{Rigidity of eigenvalues of generalized Wigner matrices}, Adv. Math. {\bf 229}(3), 1435-1515 (2012).



\bibitem{Forrester} Forrester, Peter J., and Allan K. Trinh. : \emph{Functional form for the leading correction to the distribution of the largest eigenvalue in the GUE and LUE}, Journal of Mathematical Physics {\bf 59}(5), 053302 (2018).




\bibitem{moment} He, Y., Knowles, A.:\emph{Mesoscopic eigenvalue statistics of Wigner matrices}, Ann.\ Appl.\ Probab.\ {\bf 27}(3), 1510-1550 (2017).

\bibitem{He+Knowles} He, Y., Knowles, A.: \emph{Fluctuations of extreme eigenvalues of sparse Erd\H{o}s-R\'{e}nyi graphs}, Preprint, arXiv:2005.02254, (2020).

\bibitem{He} He., Y: \emph{Spectral gap and edge universality of dense random regular graphs}, arXiv:2203.07317 (2022).

\bibitem{HLY} Huang, J., Landon, B., Yau, H.-T.: \emph{Transition from Tracy–Widom to Gaussian fluctuations of extremal eigenvalues of sparse Erdős–Rényi graphs}, Ann.\ Probab.\ \textbf{48.2} 916-962 (2020).


\bibitem{sparse_huang} Huang, J., Yau, H. T.: \emph{Edge Universality of Sparse Random Matrices}, arXiv preprint arXiv:2206.06580  (2022).



\bibitem{Ji_Oon} Hwang, J. Y., Lee, J. O., Yang, W.: \emph{Local law and Tracy–Widom limit for sparse stochastic block models}, Bernoulli, 26(3), 2400-2435 (2020).

\bibitem{gaussian_speed} Johnstone, I.\ M., Ma, Z.: \emph{Fast approach to the Tracy--Widom law at the edge of GOE and GUE}, Ann.\ Appl.\ Prob.\ {\bf 22}(5), 1962-1988 (2012).



\bibitem{KKP} Khorunzhy, A., Khoruzhenko, B., Pastur, L.: \emph{Asymptotic Properties of Large Random Matrices with Independent Entries},  J.\ Math.\  Phys.\ \textbf{37(10)}, 5033-5060 (1996).


\bibitem{KY_deform} Knowles, A., Yin, J.: \emph{Anisotropic local laws for random matrices}, Probability Theory and Related Fields, 169(1), 257-352 (2017).


\bibitem{fixed_DBM} Landon, B., Sosoe, P., Yau, H.-T.: \emph{Fixed energy universality of Dyson Brownian motion}, Adv.\ Math.\ {\bf 346}, 1137-1332 (2019).


\bibitem{LandonYau} Landon B., Yau, H.-T.: \emph{Convergence of Local Statistics of Dyson Brownian Motion},  Comm.\ Math. Phys. {\bf 355}(3), 949-1000 (2017).

\bibitem{LandonYau_edge} Landon, B, Yau, H-T.: \emph{Edge statistics of Dyson Brownian motion}, preprint, arXiv 1712.03881, 2017.



\bibitem{sparse_lee} Lee, J.: \emph{Higher order fluctuations of extremal eigenvalues of sparse random matrices}, arXiv preprint arXiv:2108.11634, 2021.

\bibitem{deformed} Lee, J.\ O., Schnelli, K.: \emph{Edge universality for deformed Wigner matrices}, Rev.\ Math.\ Phys.\ {\bf 27}(8), (2015).	


\bibitem{sample_kevin} Lee, J.\ O., Schnelli, K.: \emph{Tracy--Widom Distribution for the Largest Eigenvalue of Real Sample Covariance Matrices with General Population}, Ann.\ Appl.\ Probab.\ \textbf{26(6)}, 3786-3839 (2016).

\bibitem{sparse} Lee, J.\ O., Schnelli, K.: \emph{Local law and Tracy--Widom limit for sparse random matrices}, Probab. Theory Related Fields {\bf 171}(1), 543-616 (2018).		 

\bibitem{LY} Lee, J.\ O., Yin, J.: \emph{A Necessary and Sufficient Condition for Edge Universality of Wigner Matrices},  Duke Math.\ J.\ \textbf{163(1)}, 117-173, (2014). 



\bibitem{Lei} Lei, J.: \emph{A goodness-of-fit test for stochastic block models}, The Annals of Statistics, 44(1), 401-424 (2016).


\bibitem{LP} Lytova, A., Pastur, L.: \emph{Central Limit Theorem for Linear Eigenvalue Statistics of Random Matrices with Independent Entries}, Ann.\ Probab.\ \textbf{37}, 1778-1840 (2009).


\bibitem{Mehta} Mehta, M. L.: \emph{Random matrices}. Elsevier (2004).

\bibitem{peter2} Rahman, A. A., Forrester, P. J.: \emph{Linear differential equations for the resolvents of the classical matrix ensembles}, Random Matrices Theory Appl. {\bf 10}(03), 2250003 (2021).



\bibitem{peche}  P\'ech\'e, S.: \emph{Universality results for the largest eigenvalues of some sample covariance matrix ensembles}, Probab. Theory Relat. Fields {\bf 143}(3), 481-516 (2009).


\bibitem{PeSo2} P\'{e}ch\'{e}, S., Soshnikov, A.: \emph{Wigner Random Matrices with Non-Symmetrically Distributed Entries}, J.\ Stat.\ Phys.\ \textbf{129}, 857–884 (2007).


\bibitem{PY1}Pillai, N., Yin, J.: \emph{Universality of covariance matrices}, Ann.\ Appl.\ Probab.\ \textbf{24}(3), 935-1001 (2014).


\bibitem{Schnelli+Xu} Schnelli, K., Xu, Y.: \emph{Convergence rate to the Tracy--Widom laws for the largest eigenvalue of Wigner matrices}, Commun.\ Math.\ Phys.\ \textbf{393}, 839-907 (2022).

\bibitem{SXsample} Schnelli, K., Xu, Y.: \emph{Convergence rate to the Tracy--Widom laws for the largest eigenvalue sample covariance matrices},  arXiv:2102.04330 (2021), to appear in Ann.\ Appl.\ Prob.



\bibitem{SinaiSo} Sinai, Y. G., Soshnikov, A. B.: \emph{A refinement of Wigner's semicircle law in a neighborhood of the spectrum edge for random symmetric matrices}, Functional
Anal.\ and Appl.\ {\bf 32}, 114-131 (1998).

\bibitem{band_sodin}
Sodin, S.:. \emph{The spectral edge of some random band matrices}, Annals of mathematics, 2223-2251, (2010).


\bibitem{So1} Soshnikov, A.: \emph{Universality at the Edge of the Spectrum in Wigner Random Matrices}, Commun.\ Math.\ Phys.\ \textbf{207}, 697-733 (1999).


\bibitem{TV1}    Tao, T., Vu, V.: \emph{Random matrices: universality of local eigenvalue statistics}, Acta mathematica, 206(1), 127-204  (2011).

\bibitem{TV2} Tao, T., Vu, V.: \emph{Random Matrices: Universality of Local Eigenvalue Statistics up to the Edge}, Commun.\ Math.\ Phys.\ \textbf{298}, 549-572 (2010).

\bibitem{TW1} Tracy, C., Widom, H.: \emph{Level-Spacing Distributions and the Airy Kernel}, Commun.\ Math.\ Phys.\ {\bf 159}, 151-174 (1994).

\bibitem{TW2} Tracy, C, Widom, H.: \emph{On Orthogonal and Symplectic Matrix Ensembles}, Commun. Math. Phys.\ {\bf 177}, 727-754 (1996).


\bibitem{wang_ke} Wang, K.: \emph{Random covariance matrices: Universality of local statistics of eigenvalues up to the edge}, Random Matrices: Theory and Applications {\bf 1}(01), 1150005 (2012).

















		    
		    
				 
\end{thebibliography}
\end{document}